\documentclass{article}
\usepackage[a4paper,width=150mm,top=25mm,bottom=25mm]{geometry}

\usepackage{amsthm}%

\usepackage{authblk}
\usepackage{graphicx}%
\usepackage{multirow}%
\usepackage{amsmath,amssymb,amsfonts}%
\usepackage{hyperref}
\usepackage[capitalize]{cleveref} 
\usepackage{mathrsfs}%
\usepackage[title]{appendix}%
\usepackage[dvipsnames]{xcolor}%
\usepackage{textcomp}%
\usepackage{manyfoot}%
\usepackage{booktabs}%

\usepackage{orcidlink}

\usepackage[section]{placeins}
\let\Oldsection\section
\renewcommand{\section}{\FloatBarrier\Oldsection}
\let\Oldsubsection\subsection
\renewcommand{\subsection}{\FloatBarrier\Oldsubsection}
\let\Oldsubsubsection\subsubsection
\renewcommand{\subsubsection}{\FloatBarrier\Oldsubsubsection}

\usepackage{wrapfig}
\usepackage{pgf,tikz}
\usepackage{subcaption}
\usepackage{caption}
\usepackage{tikz-cd} 
\usepackage{floatrow} 
\usepackage{adjustbox}
\usetikzlibrary{arrows}

\graphicspath{{./}} 
\makeatletter
\def\input@path{{./}}
\makeatother

\usepackage{enumitem}
\usepackage{multicol}

\makeatletter
\newcommand{\newreptheorem}[2]{
    \newtheorem*{rep@#1}{\rep@title}
    \newenvironment{rep#1}[1]{
        \def\rep@title{#2 \ref*{##1}}
        \begin{rep@#1}
    }{
        \end{rep@#1}
    }
}
\makeatother

\theoremstyle{plain}%
\newtheorem{theorem}{Theorem}[section]
\newtheorem{proposition}[theorem]{Proposition}%

\newtheorem{lemma}[theorem]{Lemma}

\newtheorem*{theorem*}{Theorem}
\theoremstyle{remark}%
\newtheorem{remark}{Remark}

\theoremstyle{definition}%
\newtheorem{definition}{Definition}%

\newenvironment{delayedproof}[1]
 {\begin{proof}[\raisedtarget{#1}Proof of~\Cref{#1}]}
 {\end{proof}}
\newcommand{\raisedtarget}[1]{%
  \raisebox{\fontcharht\font`P}[0pt][0pt]{\hypertarget{#1}{}}%
}
\newcommand{\proofref}[1]{\hyperlink{#1}{Proof of~\Cref{#1}}}

\newcommand\keywords{%
\noindent\\\textbf{Keywords:}\ }
\usepackage{xparse}

\newcommand{\morph}{\mathrm{morph}}
\newcommand{\obj}{\mathrm{obj}}
\newcommand{\F}{\mathcal{F}}

\newcommand{\cat}[1]{\mathbf{#1}}

\NewDocumentCommand\injto{g}{%
	\ensuremath{\overset{\IfNoValueTF{#1}{}{#1}}{\rightarrowtail}}%
}
\NewDocumentCommand\simrel{g}{%
	\ensuremath{\overset{\IfNoValueTF{#1}{}{#1}}{\sim}}%
}
\NewDocumentCommand\equivrel{g}{%
	\ensuremath{\overset{\IfNoValueTF{#1}{}{#1}}{\equiv}}%
}

\newcommand{\mun}{^{-1}}

\newcommand{\bs}{\backslash}

\newcommand{\tx}[1]{\text{#1}}
\newcommand{\half}{\dfrac{1}{2}}

\newcommand{\R}{\mathbb{R}}

\newcommand{\N}{\mathbb{N}}

\newcommand\id{id}

\newcommand{\equi}{\Longleftrightarrow}

\DeclareMathOperator{\im}{im}	

\definecolor{MyGreen}{RGB}{0, 105, 50}
\newcommand{\add}[1]{{#1}}
\newcommand{\addd}[1]{{#1}}

\newcommand\subsetsim{\mathrel{%
  \ooalign{\raise0.2ex\hbox{$\subset$}\cr\hidewidth\raise-0.8ex\hbox{\scalebox{0.9}{$\sim$}}\hidewidth\cr}}}

\NewDocumentCommand\V{mggg}{%
	\ensuremath{\begin{pmatrix}#1\IfNoValueTF{#2}{}{\\#2\IfNoValueTF{#3}{}{\\#3\IfNoValueTF{#4}{}{\\#4}}}\end{pmatrix} }
}

\NewDocumentCommand\twocol{mmgg}{%
	\begin{columns}
		\begin{column}{\IfNoValueTF{#3}{0.5}{#3}\textwidth}
			#1
		\end{column}
		\begin{column}{\IfNoValueTF{#4}{0.5}{#4}\textwidth}  
			#2
		\end{column}
	\end{columns}
}
\NewDocumentCommand\threecol{mmmggg}{%
	\begin{columns}
		\begin{column}{\IfNoValueTF{#4}{0.32}{#4}\textwidth}
			#1
		\end{column}
		\begin{column}{\IfNoValueTF{#5}{0.32}{#5}\textwidth}
			#2
		\end{column}
		\begin{column}{\IfNoValueTF{#6}{0.32}{#6}\textwidth}
			#3
		\end{column}
	\end{columns}
}

\NewDocumentCommand\twocolreport{mmgg}{%
	\begin{minipage}{\textwidth}
		\begin{minipage}[!t]{\IfNoValueTF{#3}{0.5}{#3}\textwidth}
			#1
		\end{minipage}
		\begin{minipage}[!t]{\IfNoValueTF{#4}{0.5}{#4}\textwidth}
			#2
		\end{minipage}%
	\end{minipage}
}

\begin{document}

\title{Stability and Extension of Steady and Ranging Persistence}
\author[1]{{Yann-Situ} {Gazull} \thanks{yann-situ.gazull(at)univ-amu.fr}\orcidlink{0009-0003-8539-9275}}
\affil[1]{{Aix Marseille Univ, CNRS, LIS}, {{Marseille}, {France}}}
\date{November 2025}
\maketitle

\abstract{Persistent homology is a topological data analysis tool that has been widely generalized, extending its scope beyond the field of topology.
Among its extensions, steady and ranging persistence were developed to study a wide variety of graph properties.
Precisely, given a feature of interest on graphs, it is possible to build two types of persistence (steady and ranging persistence) that follow the evolution of the feature along graph filtrations.
This study extends steady and ranging persistence to other objects using category theory and investigates the stability of such persistence.
In particular, a characterization of the features that induce balanced steady and ranging persistence is provided.
The main results of this study are illustrated using a practical implementation for hypergraphs.}


\section{Introduction}
Topological data analysis has been extensively developed in recent years.
Persistent homology, one of its main tools, has become a significant research topic.
For this reason, persistent homology has been extended in several ways.

A natural extension of persistent homology is to ease the restriction that a filtration is an increasing sequence of objects.
One such extension is multidimensional persistent homology~(\cite{carlsson-multidimensional-persistence,biasotti-multidimensional-size-function}), which aims to generalize persistent homology to multidimensional filtrations and is currently a fruitful topic with applications in topological data analysis.
Another is zigzag persistence~(\cite{carlsson-zigzag}), where the maps between two consecutive spaces in the filtration can point in either direction.

A theoretical generalization of persistent homology was to theoretically develop the concept of persistence.
This induced the generalization of persistence using category theory~(\cite{bubenik-categorification-persistent-homology,bubenik-metrics-generalized-persistence,patel-generalized-diagrams}).
Within this new paradigm, concepts related to homology have been replaced by other abstract categorical concepts, and the focus has moved from homology to the persistence function itself.
From this viewpoint, persistence is seen as a process that takes a filtration and produces a persistence diagram by analyzing how a specific concept persists along the filtration.

Another extension was driven by the will to use a persistence diagram on objects other than topological sets and for features other than homology, typically to study certain properties of graphs.
An approach is to build a topological complex whose homology corresponds to the characteristics studied (see, for example,~\cite{grigoryan-homology-digraph-2021,grigoryan-path-complexes-2020,grigoryan-graphs-simplical-complexes-2014,bressan-homology-hypergraph}).
Another approach is to use the aforementioned categorical persistence framework to study persistent features without using homology and topological objects.
An example is the study of persistent hypergraph attributes \add{in~\cite{liu-persistent-spectral-hypergraph-learning,liu-hypergraph-homology} and in~\cite{babu-persistent-hypergraph}}. 
Rank-based persistence~(\cite{bergomi-rank-based-persistence}) and graph persistence~(\cite{bergomi-exploring-graph-persistence}) were also developed in this context, as well as steady and ranging persistence~(\cite{bergomi-steady-ranging}), which constitutes the starting point of the present work.

Steady and ranging persistence is a method developed by Bergomi et al. to study the evolution of features over graphs (or digraphs) along a filtration.
Specifically, given a feature on graphs (e.g., Eulerian sets or hubs in~\cite{bergomi-steady-ranging}), they define two ways of studying the persistence of this feature, namely steady and ranging persistence.
In their work, they presented some examples of graph features and studied the stability of steady and ranging persistence induced by these features.
\smallskip

Our theoretical work is twofold.
First, we use category theory to extend the steady and ranging persistence paradigm to a larger variety of objects (e.g., hypergraphs).
\add{Second, we address the problem of stability in the context of steady and ranging persistence.
Precisely, we extend the notion of balancedness and discuss its relationship with stability.
We then provide a characterization of features that induce balanced steady (respectively, ranging) persistence (see our main result~\Cref{coro:convex-balanced-equivalence}).}

The remainder of this paper is organized as follows.
In~\Cref{sect:prior-works}, a mathematical background is introduced.
Precisely, we present categorical persistence, ip-generators and steady and ranging persistence for graphs.
Then we extend the scope of steady and ranging persistence in~\Cref{sect:extending-steady-ranging}.
\add{In particular, we investigate to what extent balancedness is equivalent to stability in this context.}
In~\Cref{sect:convex-balanced}, convex features are introduced, and our main result (which is the characterization of balanced features) is established.
Finally, in~\Cref{sect:examples-constructions-hypergraphs}, we present experiments on hypergraphs to illustrate our two theoretical contributions.

\section{Related Works}\label{sect:prior-works}
This section introduces previous works that are important for the present work, and is organized as follows:
we first introduce the categorical persistence framework, then provide useful definitions concerning the stability of persistence, and finally, we present the steady and ranging persistence theory for graphs, which corresponds to the starting point of our work.

\subsection{Categorical persistence}
    In the categorical persistence framework, category theory is used to define the generalization of filtrations (which are no longer growing topological objects), persistence functions (which correspond to persistence diagrams), and ip-generators (which are functions that associate a filtration with a persistence diagram).
    
    Categorical persistence has been formalized in various studies (see, for example,~\cite{bubenik-categorification-persistent-homology}). Here we follow the formalism used in~\cite{bergomi-steady-ranging}.
    In this study, we do not introduce category theory.
    See~\cite{vanoosten-basic-category-theory,leinster-basic-category-theory} for an introduction to category theory.
    
    In the present work, we mainly consider categories where every morphism is a monomorphism.
    Such categories will be called \emph{mono categories}.
    Given a category $\cat{C}$, we denote $\cat{C}_m$ as the category whose objects are the objects of $\cat{C}$ but whose morphisms are the monomorphisms of $\cat{C}$.
    In particular, we write $\R_m$ for the category whose objects are real numbers and arrows correspond to the $\leq$ relation (usually denoted $(\R,\leq)$).
    We use the notation $X\injto{\iota}X'$ for a monomorphism $\iota$ from $X$ to $X'$.
    \smallskip
    
    \begin{definition}[persistence function~{\cite[Definition 1]{bergomi-steady-ranging}}]\label{def:persistence-function}
        A \emph{persistence function} is a function $p:\morph(\R_m)\to \N$ that satisfies the following properties:
        for every $u_1\leq u_2\leq v_1\leq v_2$:
        \begin{enumerate}
            \item $p(u_1\leq v_1)\leq p(u_2\leq v_1)$
            \item $p(u_2\leq v_2)\leq p(u_2\leq v_1)$
            \item $p(u_2\leq v_1)-p(u_1\leq v_1)\geq p(u_2\leq v_2)-p(u_1\leq v_2)$
        \end{enumerate}
    \end{definition}
    Persistence functions are as informative as persistence diagrams.
    Indeed, persistence functions give rise to - and can be recovered from - persistence diagrams (see~\Cref{fig:persistence-diagram-example}).
    See~\cite{cohen-steiner-stability-persistence} for the formal definition of a persistence diagram.
    \addd{By convention, we set $p(-\infty\leq v) := \lim_{u\to -\infty}p(u \leq v)$ and $p(u \leq +\infty) := \lim_{v\to +\infty}p(u \leq v)$, which are well-defined by \emph{1.} and \emph{2.} of~\Cref{def:persistence-function}.}
    To recover the multiset definition of a persistence diagram, we can define the multiplicity as follows:
    \begin{definition}[{\cite{bergomi-rank-based-persistence}}, Definition 5]\label{def:persistence-multiplicity}
      Given $u < v \in \R \cup \{-\infty, +\infty\}$ we define the multiplicity $\mu(u,v)$ of $(u, v)$ as the minimum of the following expression, over $I_u$ , $I_v$ disjoint connected neighborhoods of $u$ and $v$ respectively:
      $$p(\sup I_u\leq \inf I_v) - p(\inf I_u\leq \inf I_v) -  p(\sup I_u\leq \sup I_v) + p(\inf I_u\leq \sup I_v)$$
      A point $(u,v)$ is said to be a \textit{cornerpoint} of $p$ if $\mu(u,v)>0$.
      By convention we also define $\mu(u,u) = +\infty$.
    \end{definition}
    \addd{\begin{remark}\label{rem:finite-multiplicities}
      The quantity in~\Cref{def:persistence-multiplicity} is increasing in both $I_u$ and $I_v$ (with respect to inclusion), so the minimum is achieved for sufficiently small intervals $I_u$ and $I_v$.
      In addition,~\Cref{def:persistence-function} prevents $p(u\leq v)$ from taking infinite values.
      It follows that every cornerpoint outside of the diagonal $y=x$ has finite multiplicity.
    \end{remark}}
    
    \begin{definition}[bottleneck distance,~\cite{chazal-proximity-persistence-module}]\label{def:bottleneck-distance}
        Given two persistence functions $p$ and $p'$, we denote $Dgm(p)$ and $Dgm(p')$ their corresponding diagram seen as multisets and $\Gamma$ the set of all bijections from $Dgm(p)$ to $Dgm(p')$.
        The \emph{bottleneck distance} $d_B(p,p')$ between $p$ and $p'$ of is defined as
        $$ d_B(p,p') = \inf_{\gamma\in\Gamma}\sup_{c\in Dgm(p)}||c - \gamma(c) ||_\infty $$
    \end{definition}
    \begin{figure}[!htb]
        \centering
        \includegraphics[width = 0.32\textwidth]{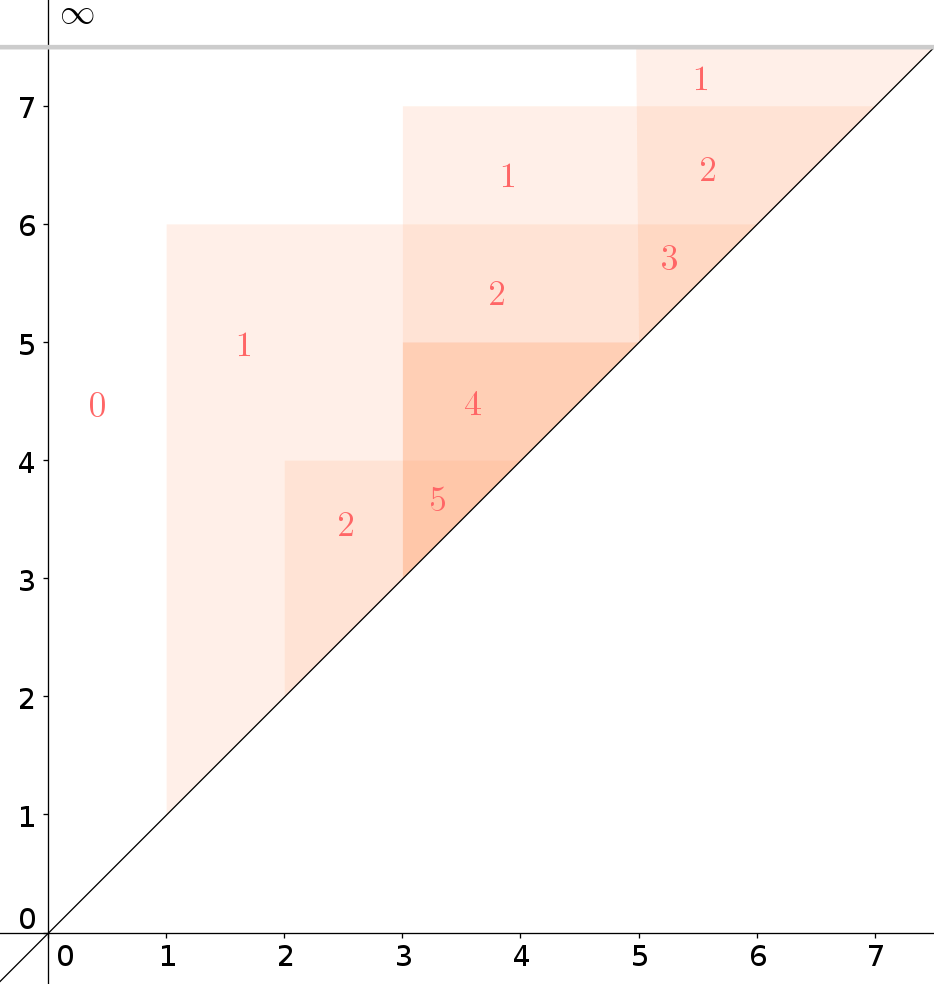}
        \hspace{0.12\textwidth}
        \includegraphics[width = 0.32\textwidth]{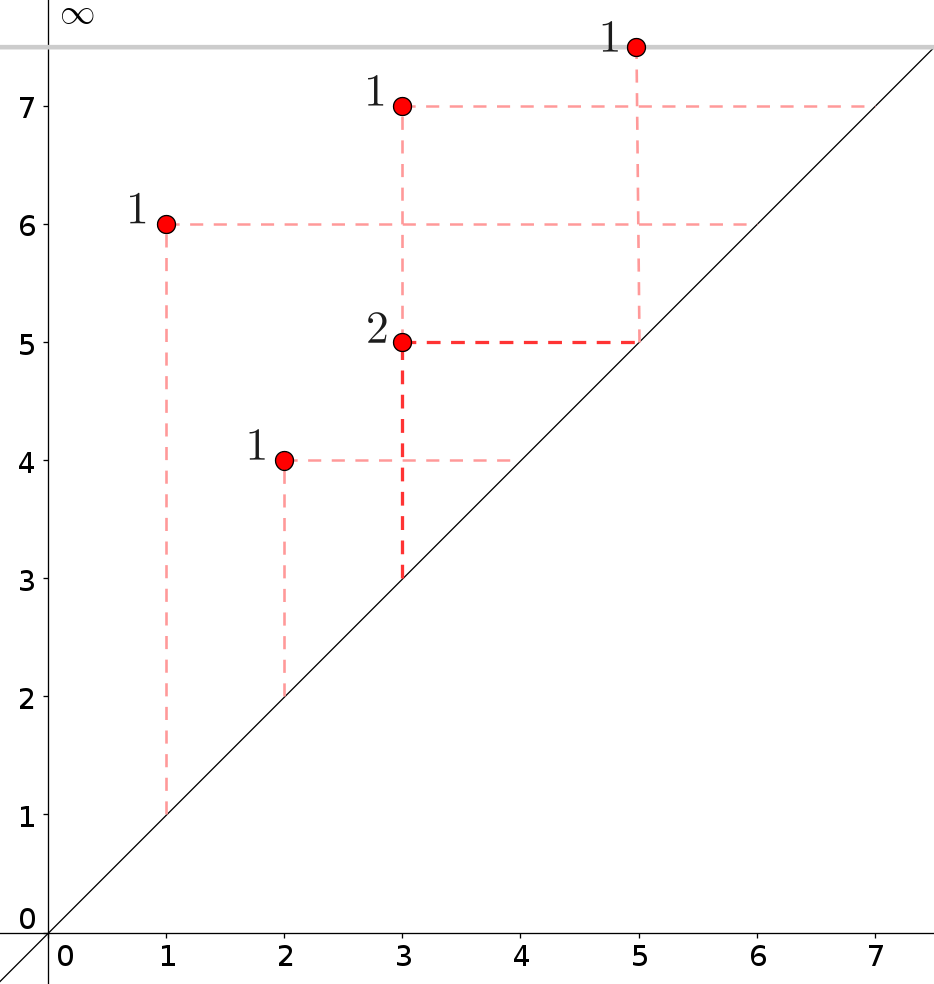}
          \caption{\add{A persistence function (on the left) and its representation as a persistence diagram (on the right).
          \addd{In both images, a point $(u,v)$ (above the diagonal $y=x$) corresponds to an element $(u\leq v)$ of $\morph(\R_m)$.}
          In the right diagram, the multiplicity is written above each cornerpoint.}}
          \label{fig:persistence-diagram-example}
    \end{figure}
    Filtration is an important component of persistence.
    In a persistent homology setting, filtrations are growing sequences of topological objects.
    In this categorical setting, filtrations are defined as functors from ordered real numbers to the category of interest.
    \begin{definition}[filtration]
       A \emph{filtration} of a mono category $\cat C_m$ is a functor $F:\R_m \to\cat C_m$.
       For every $u,v\in\R$, we will write $F_u:=F(u)$ and $F_{u}^{v}:= F(u\leq v)$.
       Note that we have $F_{v}^{w}F_{u}^{v}=F_{u}^{w}$.   
    \end{definition}
    \begin{remark}
        In~\cite{bubenik-categorification-persistent-homology}, filtrations are referred to as $(\R,\leq)$-diagrams.
    \end{remark}
    \smallskip
    
    \add{
    An example is the filtration built from a weighted graph:
    \begin{definition}[weighted graph filtration]\label{def:weighted-graph}
        Consider the category $\cat{Gph}_m$ of undirected graphs with monomorphisms.
        A \textit{weighted graph} is a pair $(G,f)$ where $G = (V,E)$ is a graph (with $V$ the vertices and $E$ the edges) and $f$ is a function from $V\sqcup E$ to $\R$ such that $f(v)\leq f(e)$ whenever $v$ is a vertex of the edge $e$.
        
        A weighted graph induces a filtration $F$ of $\cat{Gph}_m$:
        \begin{align*}
        F:\ \quad u\quad &\mapsto\ F_u\ :=\ \left(f^{-1}(]-\infty,u])\cap V,\ f^{-1}(]-\infty,u])\cap E\right) \\
         (u\leq v) &\mapsto (F_u \subseteq F_v)
        \end{align*}
        where $(F_u \subseteq F_v)$ is the morphism that maps each vertex (or edge) in $F_u$ to the same vertex (or edge) in $F_v$.
    \end{definition}
    }

    \begin{definition}[ip-generator~\add{\cite[Definition 5]{bergomi-steady-ranging}}]\label{def:ip-generator}
       An \emph{ip-generator} (\textit{indexed-aware persistence function generator}) $\pi$ for $\cat C_m$ is a function that assigns a persistence function $\pi(F)$ to every filtration $F$ of $\cat C_m$\add{, such that $\pi(F) = \pi(F')$ whenever there exists a natural isomorphism between $F$ and $F'$ (see~\cite{leinster-basic-category-theory} for the definition of a natural isomorphism).}
    \end{definition}

    \paragraph{Example}
    A common type of ip-generator is obtained using the notion of categorical persistence function~{\cite[Definition 3]{bergomi-rank-based-persistence}}:
        a \emph{categorical persistence function} of the category  $\cat{C}_m$ is a function $p:\morph(\cat{C}_m)\to \N$ that satisfies the following properties:
        for every $u_1\injto u_2\injto v_1\injto v_2$:
        \begin{enumerate}
            \item $p(u_1\injto v_1)\leq p(u_2\injto v_1)$
            \item $p(u_2\injto v_2)\leq p(u_2\injto v_1)$
            \item $p(u_2\injto v_1)-p(u_1\injto v_1)\geq p(u_2\injto v_2)-p(u_1\injto v_2)$
        \end{enumerate}

    Given a categorical persistence function $p$ for $\cat C_m$ it is possible to define its corresponding ip-generator $\pi_p:F\mapsto p \circ F$, which takes a filtration $F$ of $\cat C_m$ and provides the persistence function $p\circ F$.
    Hence, the persistent homology setting of dimension $q$ can be recovered in $\cat C_m = (\cat{Top}, \subseteq)$ (the category of topological objects with inclusions), using the ip-generator $\pi_{h_q}:F\mapsto h_q \circ F$ obtained  from the following categorical persistence function:
    $$h_q: (F_u\subseteq F_v) \mapsto dim\big( im \left(H_q(F_u) \rightarrow H_q(F_v) \right)\big)$$
    where $F_u$ and $F_v$ are topological sets (with finitely generated homology groups) and $H_q(F_u) \rightarrow H_q(F_v)$ is the homology group morphism induced by the inclusion $F_u\subseteq F_v$.
    See~\cite{bubenik-categorification-persistent-homology} for more details on how to recover standard persistent homology with categorical formalism.
    See~\cite{bergomi-rank-based-persistence} for a generalization of the construction of categorical persistent functions similar to $h_q$, using the notion of \textit{rank functions}.

\subsection{Stability of persistence}
    One of the main advantages of persistent homology is that persistent diagrams are stable.
    The standard stability theorem was presented for persistent homology in~\cite{cohen-steiner-stability-persistence} and generalized in~\cite{chazal-proximity-persistence-module}:
    \begin{theorem}[stability of persistent homology,~{\cite{cohen-steiner-stability-persistence}}]\label{thm:stability-persistent-homology}
        Given $X$ a triangulable space with continuous tame functions $f$ and $g$ from $X$ to $\R$, we have:
        $$d_B (Dgm(f), Dgm(g)) \leq ||f - g||_\infty$$
        where $d_B$ is the bottleneck distance (or matching distance) and $Dgm(f)$ and $Dgm(g)$ are the persistence diagrams obtained by the filtrations induced by $f$ and $g$.
    \end{theorem}

    We now define the stability within the present categorical persistence setting using works from~\cite{bergomi-rank-based-persistence,bergomi-steady-ranging,amico-natural-pseudo-distance,chazal-proximity-persistence-module,lesnick-theory-interleaved-distance-multidimensional}.
    First, we give the definition of the \textit{interleaving distance} between two filtrations.
    As pointed out in~\cite{bergomi-beyond-topological-persistence}, the interleaving distance is equal to the \textit{natural pseudo distance} of~\cite{amico-natural-pseudo-distance}.
    \begin{definition}[$\epsilon$-interleaved filtrations,~\cite{chazal-proximity-persistence-module,bubenik-categorification-persistent-homology,bubenik-metrics-generalized-persistence,lesnick-theory-interleaved-distance-multidimensional}]\label{def:epsilon-close}
       Two filtrations $F$ and $G$ of $\cat C_m$ are said to be \emph{$\epsilon$-interleaved} if there exist two families of monomorphisms $(\phi_w:F_w\to G_{w+\epsilon})_{w\in\R}$ and $(\psi_w:G_w\to F_{w+\epsilon})_{w\in\R}$ such that for any $u,v,w\in\R$, the following diagrams commute:
       \begin{center}
\adjustbox{scale=0.95}{%
\begin{tikzcd}
F_{w-\epsilon} \arrow[rd, "\phi_{w-\epsilon}" description, shift right] \arrow[rr, "F_{w-\epsilon}^{w+\epsilon}"] &                                                   & F_{w+\epsilon}                                                       & F_{u+\epsilon} \arrow[r, "F_{u+\epsilon}^{v+\epsilon}"]              & F_{v+\epsilon} \\
                                                                                                                              & G_w \arrow[ru, "\psi_w" description]              & G_u \arrow[ru, "\psi_u" description, shift left] \arrow[r, "G_u^v"]  & G_v \arrow[ru, "\psi_v" description]                                 &                \\
                                                                                                                              & F_w \arrow[rd, "\phi_w" description, shift right] & F_u \arrow[rd, "\phi_u" description, shift right] \arrow[r, "F_u^v"] & F_v \arrow[rd, "\phi_v" description, shift left]                     &                \\
G_{w-\epsilon} \arrow[ru, "\psi_{w-\epsilon}" description, shift left] \arrow[rr, "G_{w-\epsilon}^{w+\epsilon}"]              &                                                   & G_{w+\epsilon}                                                       & G_{u+\epsilon} \arrow[r, "G_{u+\epsilon}^{v+\epsilon}", shift right] & G_{v+\epsilon}
\end{tikzcd}
}
       \end{center}
       
        The \emph{interleaving distance} $\delta(F,G)$ between two filtrations $F$ and $G$ of $\cat C_m$ is the infimum of the $\epsilon$ such that $F$ and $G$ are $\epsilon$-interleaved.
        It is equal to $+\infty$ if there is no $\epsilon$ such that $F$ and $G$ are $\epsilon$-interleaved.
    \end{definition}

    \begin{definition}[stable ip-generators]\label{def:stability-ip-generators}
        An ip-generator $\pi$ for $\cat C_m$ is said to be \emph{stable} if for every filtrations $F$ and $G$ of $\cat C_m$:
        $$ d_B(\pi(F), \pi(G)) \leq \delta(F, G) $$
    \end{definition}
    In particular, it was proven in~\cite{chazal-proximity-persistence-module} (Theorem 4.4) that the persistent homology ip-generator $\pi_{h_q}$ is stable.

    \smallskip

    In~\cite{bergomi-steady-ranging}, it has been proven using works from~\cite{amico-natural-pseudo-distance} that the stability of an ip-generator $\pi$ \add{for weighted graphs} is equivalent to $\pi$ being ``balanced'' (see~\Cref{prop:stability-balanced}).
    We present here a generalization of the definition of a balanced ip-generator (see~\Cref{def:balanced-generator}).
    Specifically, the definition of $\epsilon$-interleaved filtrations is used to extend the definition of balanced ip-generators from graphs to arbitrary categories.
    
    \begin{definition}[$\epsilon$-compatible persistence functions]\label{def:epsilon-compatible}
       Two persistence functions $p$ and $p'$ are said to be \emph{$\epsilon$-compatible} if for every $(u\leq v)$ we have $p(u-\epsilon\leq v+\epsilon) \leq p'(u\leq v)$ and $p'(u-\epsilon\leq v+\epsilon) \leq p(u\leq v)$.
    \end{definition}

    \begin{definition}[balanced ip-generator]\label{def:balanced-generator}
        An ip-generator $\pi$ is said to be \emph{balanced} if for every $\epsilon$-interleaved $F$ and $G$, $\pi(F)$ and $\pi(G)$ are $\epsilon$-compatible.
    \end{definition}

    For weighted graphs, this definition coincide with the notion of balancedness introduced in~\cite{bergomi-steady-ranging}.
    In particular, the following lemma shows that the standard property of balanced ip-generators for weighted graph filtrations (as stated in~\cite[Definition 6]{bergomi-steady-ranging}) is equivalent to~\Cref{def:epsilon-close}.
    \begin{lemma}[$\epsilon$-interleaved filtration for graphs (see~{\cite[Definition 6]{bergomi-steady-ranging}})]\label{lem:epsilon-close-charac}
      Let $(X,f)$ and $(X',g)$ be two weighted graphs whose induced filtrations are respectiveley $F$ and $G$, and such that $X$ and $X'$ are isomorphic.
      Then $F$ and $G$ are $\epsilon$-interleaved if and only if there exists an isomorphism $\phi:X\to X'$ such that  $sup_{x\in X}|f(x)-g(\phi(x))|\leq \epsilon$.
    \end{lemma}
    See~\proofref{lem:epsilon-close-charac} in the Appendix.

    
    \begin{proposition}[{\cite[Theorem 1]{bergomi-steady-ranging}}]\label{prop:stability-balanced}
        \add{If $\pi$ is an ip-generator that is balanced for weighted graph filtrations, then it is stable for weighted graph filtrations.}
    \end{proposition}
    \add{
    Unfortunately,~\Cref{prop:stability-balanced} is hard to generalize in a wider context.
    This problem is discussed in~\Cref{sect:stability-balancedness}.
    }

\subsection{Steady and ranging persistence for Graphs}
    Now that persistence is freed from its homology chains, categorical persistence leads to the development of persistence for specific types of data, such as graph persistence.
    We consider the category $\cat{Gph}_m$ of graphs with monomorphisms (note that the theory presented in this section also works with the category of digraphs).

    Different methods of building persistence functions for graphs have been developed.
    In~\cite{bergomi-beyond-topological-persistence}, the authors defined a method for building interesting graph persistence functions using \textit{posets} and the notion of weakly directed properties.
    Another method of building a graph persistence function was introduced in~\cite{bergomi-steady-ranging} through the theory of steady and ranging graph persistence.
    We introduce the main definitions and results of this theory, as it is the starting point of the present work.
    \begin{definition}[graph feature,~{\cite[Definition 8]{bergomi-steady-ranging}}]\label{def:graph-feature}
        \addd{A \emph{graph feature} $\F$ consists, for every graph $G=(V,E)$, of a function $\F_G:\mathcal{P}(V\cup E)\mapsto\{true,false\}$ (where $\mathcal{P}(V\cup E)$ is the power set of $V\cup E$).
        
        Any $X\subseteq V\cup E$ such that $\F_G(X)=true$ is called an \emph{$\F$-set} of $G$.
        If $(G,f)$ is a weighted graph with induced filtration $F$ (see~\Cref{def:weighted-graph}), $X$ is said to be an $\F$-set at level $w \in \R$ if $\F_{F_w}(X)=true$.}
    \end{definition}

    \begin{definition}[graph steady and ranging sets,~{\cite[Definition 10]{bergomi-steady-ranging}}]\label{def:graph-steady-ranging-sets}
        \add{Let $(G,f)$ be a weighted graph.
        $X$ is a \emph{steady $\F$-set} at $(u\leq v)$ if it is an $\F$-set at every level $w\in[u,v]$.\\
        $X$ is a \emph{ranging $\F$-set} at $(u\leq v)$ if it is an $\F$-set at a level $w\leq u$ and at a level $w' \geq v$.
        We respectively denote $S^\F_{(G,f)}(u\leq v)$ and $R^\F_{(G,f)}(u\leq v)$ the set of steady and ranging $\F$-sets at $(u\leq v)$.}
    \end{definition}

    \begin{definition}[graph steady and ranging generators,~{\cite[Definition 11]{bergomi-steady-ranging}}]\label{def:graph-steady-ranging-generators}
        \add{Given a weighted graph $(G,f)$, we define the two following maps:
        $$\sigma^\F_{(G,f)}:(u\leq v) \mapsto |S^\F_{(G,f)}(u\leq v)| \qquad \rho^\F_{(G,f)} : (u\leq v) \mapsto |R^\F_{(G,f)}(u\leq v)|$$
        They naturally induce the following functions, which are respectively called the \emph{steady and ranging generators} of the feature $\F$:
        $$\sigma^\F : (G,f)\mapsto \sigma^\F_{(G,f)} \qquad \rho^\F : (G,f)\mapsto \rho^\F_{(G,f)}$$}
    \end{definition}
    \begin{proposition}[{\cite{bergomi-steady-ranging}}, Proposition 2]\label{prop:graph-steady-ranging-generators}
        For every $(G,f)$, $\sigma^\F_{(G,f)}$ and $\rho^\F_{(G,f)}$ are persistence functions, and $\sigma^\F$ and $\rho^\F$ are ip-generators.
    \end{proposition}
    \smallskip
    
    Consequently, every graph feature induces two different ip-generators, which can be used to study the persistence of this feature.
    As highlighted by~\cite{bergomi-steady-ranging}, not every steady and ranging generator is balanced.
    Indeed, an interesting problem is to understand which features induce a balanced steady or ranging generator.
    In the same paper, the authors defined a class of features called \textit{monotone features}, which induce balanced steady and ranging generators.
    Similarly, in~\cite{bergomi-exploring-graph-persistence}, a larger class of features inducing balanced steady and ranging generators was defined, namely \textit{simple features}.
    These definitions and results are stated below:

    \begin{definition}[monotone graph feature~{\cite[Definition 13]{bergomi-steady-ranging}}]\label{def:graph-monotone-feature}
        A graph feature $\F$ is said to be \emph{monotone} if it satisfies the following two properties:
        \addd{
        \begin{enumerate}
            \item for any (di)graph $G = (V, E) \subseteq G' = (V' , E')$ and any $X \subseteq (V\cup E)$, we have $\F_{G'}(X) = true \implies$ $\F_G(X) = true$.
            \item For any (di)graph $G = (V, E)$ and any $Y\subseteq X \subseteq V \cup E$, we have $\F_G(X) = true \implies \F_G(Y) = true$.
        \end{enumerate}
        }
    \end{definition}

    \begin{proposition}[{\cite{bergomi-steady-ranging}}, Propositions 3 and 4]\label{prop:graph-monotone-feature}
        If $\F$ is a monotone graph feature, then $\sigma^\F = \rho^\F$ is a balanced ip-generator.
    \end{proposition}

    \begin{definition}[simple graph feature~{\cite[Definition 17]{bergomi-exploring-graph-persistence}}]\label{def:graph-simple-feature}
        A graph feature $\F$ is said to be \emph{simple} if it satisfies the following property:
        for any (di)graphs $H_1 \subseteq H_2 \subseteq H_3 \subseteq G$, \addd{$\F_{H_3}(X) = true$ and $\F_{H_2}(X) = false$ implies $\F_{H_1}(X) = false$.}
    \end{definition}

    \begin{proposition}[{\cite{bergomi-exploring-graph-persistence}}, Propositions 20, 24 and 26]\label{prop:graph-simple-feature}
        If $\F$ is a simple graph feature, then $\sigma^\F = \rho^\F$ are balanced ip-generators.
        Moreover, every monotone feature is simple.
    \end{proposition}

    Our main contribution is to extend the steady and ranging persistence framework to other categories and to characterize the features that provide balanced steady and ranging generators.
    A corollary of our main result is that the converse of~\Cref{prop:graph-simple-feature} is also true: if $\sigma^\F$ and $\rho^\F$ are balanced ip-generators then $\F$ is simple.

\section{Extending Steady and Ranging Persistence}\label{sect:extending-steady-ranging}
    In this section we aim to redefine the steady and ranging persistence framework of~\cite{bergomi-steady-ranging} to extend it to a wider class of objects (e.g., hypergraphs, simplicial complexes, or matroids).
    More precisely, we define steady and ranging persistence for a wide variety of concrete categories.

    \subsection{Finitely concrete mono categories}\label{sect:concrete-category}
        In the steady and ranging definitions given in the last section, features were considered as sets of ``subsets of the elements composing a graph'' (i.e., vertices and edges).
        To generalize the concept of feature to other categories, we need to find a definition that grasps the idea of ``a subset of the elements composing $X$'', with $X$ being an object of the studied category.

        In this section, we provide a rigorous definition of this general idea.
        More precisely, we draw from the concept of concrete category to properly define the generalized concept of feature.
        A concrete category is a category equipped with a faithful functor to $\cat{Set}$ (see~\cite{leinster-basic-category-theory} for the definition of a faithful functor).
        We use a similar concept to define a class of categories suited to our context:

        %
        
        \begin{definition}[finitely concrete mono category]\label{def:finitely-concrete-category}
            A \emph{finitely concrete mono category} is a mono category $\cat{C}_m$ equipped with a faithful functor $\Phi:\cat{C}_m\to\cat{FinSet}_m$,
            where $\cat{FinSet}_m$ is the category of finite sets whose morphisms are the injective functions between sets.
            \add{For simplicity, we write $\Phi X$ instead of $\Phi(X)$ whenever no ambiguity arises.}
        \end{definition}
        \add{
        \begin{remark}
            Although monomorphisms are frequently injective in commonly used categories (for instance in hypergraph categories, see~\Cref{sect:examples-constructions-hypergraphs}), this is not a requirement for the category $\cat C_m$.
            Here, it is only required that their image under $\Phi$ is injective.
        \end{remark}
        }
        \begin{definition}[empty-like element]\label{def:empty-like-element}
            Given a finitely concrete mono category $(\cat{C_m}, \Phi)$, an \emph{empty-like element} $\hat\emptyset$ is an initial object of $\cat{C}_m$ such that $\Phi\hat\emptyset = \emptyset$.
            
            See~\cite{leinster-basic-category-theory} for the definition of an initial object.
        \end{definition}

        \begin{definition}[feature]\label{def:feature-concrete}
            Given a finitely concrete mono category $(\cat{C_m}, \Phi)$, a \emph{feature} is a set $\F$ of pairs $(A,X)$, with $X\in\obj(\cat{C_m})$ and $A\subseteq\Phi X$,
            \add{such that if $f : X \to X'$ is an isomorphism, then $(A,X)\in F \iff (\Phi f(A), X') \in \F$.}
        \end{definition}
        \smallskip
        
        Given a monomorphism $\iota:X\injto X'$, we use the following notation: $\iota(A,X) := (\Phi\iota(A), X')$.
        By functoriality, this definition is compatible with the composition of monomorphisms.
        Indeed, given $X\injto{\iota}X'\injto{\iota'}X''$ we have $\iota'\circ \iota(A,X) = \iota'\left( \iota(A, X) \right)$.
        This allows us to track a pair $(A,F_u)$ in a filtration $F$ by looking at $F_u^x(A,F_u) = (\Phi F_u^x(A), F_x)$.

        \paragraph{Example}
        A detailed example of a finitely concrete mono category with an empty-like element is given in~\Cref{sect:examples-constructions-hypergraphs} with the concept of hypergraph.
        We present here another example of such a category using the concept of simplicial complex, and we present two examples of features for this category.

            A finite simplicial complex is a finite set $K$ of simplices of arbitrary dimensions \add{living in a Euclidean space of fixed dimension} such that:
            \begin{itemize}
                \item every face of a simplex in $K$ is also in $K$;
                \item the intersection of two non-disjoint simplices of $K$ is a common face of them.
            \end{itemize}
            Let $\cat{Simp}_m$ be the mono category whose objects are finite simplicial complexes and whose morphisms are inclusions between simplicial complexes.

            The finitely concrete functor $\Phi$ is defined as follows:
            \begin{align*}
                \Phi:\quad \obj(\cat{Simp}_m)\quad &\to\quad \obj(\cat{FinSet}_m)\\
                     K \qquad &\mapsto\qquad K\\
                \Phi:\ \morph(\cat{Simp}_m)\ &\to\ \morph(\cat{FinSet}_m)\\
                    (K\subseteq K')\ &\mapsto\ \Phi(K\subseteq K'):\sigma\in K \mapsto \sigma\in K'
            \end{align*}
            First, $\Phi(K\subset K')$ is injective, so it indeed belongs to $\morph(\cat{FinSet}_m)$.
            Second, $\Phi$ is faithful, because $Hom(K,K')$ is either the empty set or the singleton $\{K\subseteq K'\}$.

            The empty-like object is the empty simplicial complex $K = \hat \emptyset = \emptyset$.
            It is initial because, for every finite simplicial complex $K$, there exists one and only one inclusion $\emptyset \subseteq K$.
            Moreover, $\Phi(\emptyset) = \emptyset$.

            Two examples of features for $\cat{Simp}_m$ are presented here:
            \begin{itemize}
                \item $\F^0$ is the ``isolated-vertex feature'', where $(A,K)$ is in $\F^0$ if and only if $A\subseteq \Phi K$ is a singleton $\{\sigma\}$ of a $0$-simplex (i.e., a vertex) that is not contained in another larger simplex of $K$.
                \item $\F^{\nabla}$ is the ``triangle feature'', where $(A,K)$ is in $\F^{\nabla}$ if and only if $A\subseteq \Phi K=K$ is a set of three $1$-simplices (i.e., segments) that form a triangle in $K$.
            \end{itemize}


        %

    \subsection{Redefining steady and ranging persistence} \label{sect:redefining-steady-ranging}

        In the remaining of this work, $\cat C_m$ is considered to be a category satisfying the following assumptions:
        \begin{enumerate}
            \item $\cat C_m$ is a \emph{finitely concrete mono category} with functor $\Phi$ (see~\Cref{def:finitely-concrete-category}).
            \item $\cat C_m$ has an \emph{empty-like object} $\hat\emptyset$ (see~\Cref{def:empty-like-element}).
        \end{enumerate}
        \smallskip
        
        \begin{definition}[steady and ranging sets]\label{def:steady-ranging-set}
        Let $F$ be a filtration of $\cat C_m$, and let $\F$ be a feature of $\cat C_m$. For every $u\leq v$, the \emph{$\F$-steady} and \emph{$\F$-ranging sets} of $F$ are defined as follows:
        \begin{align*}
        S^\F_F(u\leq v) &= \Bigl\{ (A,F_u)\ /\ \forall w\in[u,v],\ F_u^w(A,F_u)\in\F \Bigr\} \\
        R^\F_F(u\leq v) &= \left\{ (A,F_u)\ /\ \exists x\leq u,\ y\geq v\ \tx{and}\ A'\ /
        \begin{cases}
          (A',F_x)\in\F &\tx{with}\ A=\Phi F_x^u(A')\\
          F_u^{y}(A,F_{u})\in\F
        \end{cases}  \right\}
        \end{align*}
        \end{definition}

        These two definitions imply the following elementary properties:
        \begin{proposition}\label{prop:steady-ranging-basic}
            Given $u_1\leq u_2\leq v_1\leq v_2$, the following properties are true:
            \begin{align*}
                0.\qquad & S^\F_F(u_1\leq v_1) \subseteq R^\F_F(u_1\leq v_1)\\
                1S.\qquad & \addd{(A,X)\in S^\F_F(u_1\leq v_2) \longmapsto {F_{u_1}^{u_2}}(A,X) \in S^\F_F(u_2\leq v_1)\quad\tx{is injective}}\\
                2S.\qquad & S^\F_F(u_2\leq v_2) \bs S^\F_F(u_1\leq v_2) \subseteq S^\F_F(u_2\leq v_1) \bs S^\F_F(u_1\leq v_1)\\
                1R.\qquad & \addd{(A,X)\in R^\F_F(u_1\leq v_2) \longmapsto {F_{u_1}^{u_2}}(A,X) \in R^\F_F(u_2\leq v_1)\quad\tx{is injective}}\\
                2R.\qquad & R^\F_F(u_2\leq v_2) \bs R^\F_F(u_1\leq v_2) \subseteq R^\F_F(u_2\leq v_1) \bs R^\F_F(u_1\leq v_1)
            \end{align*}
        \end{proposition}
        The proof of these properties is the same as that of Lemma~1 and Proposition~2 in~\cite{bergomi-steady-ranging}.
        \begin{definition}[steady and ranging generators]\label{def:steady-ranging-generators}
            Given $\F$ a feature of $\cat C_m$, the $\F$-steady and $\F$-ranging functions for a filtration $F$ are defined as follows:
            \begin{align*}
            \sigma^\F_F: (u\leq v) &\mapsto |S^\F_F(u\leq v)|\\
            \rho^\F_F: (u\leq v) &\mapsto |R^\F_F(u\leq v)|
            \end{align*}
            $\sigma^\F_F$ and $\rho^\F_F$ are persistence functions.
            This follows from~\Cref{prop:steady-ranging-basic}(\emph{1S.}, \emph{1R.}, \emph{2S.}, and \emph{2R.}) \addd{and from the finiteness of $S_{F}^\F(u\leq v)$ and $R_{ F}^\F(u\leq v)$ for all $(u\leq v)$, as they are subsets of the finite set $\{(A, F_u), A\subseteq\Phi F_u\}$.}
            
            As a result, $\sigma^\F: F\mapsto \sigma^\F_F$ and $\rho^\F: F\mapsto \rho^\F_F$ are ip-generators, that is, the \emph{steady} and \emph{ranging generators} of $\F$. 
            \add{The invariance of $\sigma^\F_F$ and $\rho^\F_F$ with respect to natural isomorphism follows from the property that $(A,X)\in F$ if and only if $f(A, X') \in \F$, for any isomorphism $f:X\to X'$ (see~\Cref{def:feature-concrete}).}
        \end{definition}

        All in all, it is possible to build two different ip-generators from each feature.
        Next sections investigate the stability and balancedness of such generators.
  \add{
  \subsection{Stable and balanced steady and ranging generators} \label{sect:stability-balancedness}
    }
    \add{
    \Cref{prop:stability-balanced} states that balanced ip-generators for weighted graph filtrations are stable.
    Unfortunately, this proposition does not easily extend to our framework.
    This section analyses to what extent balancedness is equivalent to stability in the context of generalized steady and ranging persistence.
    We first introduce the notion of \textit{tame filtration}.
    \begin{definition}[tame filtration~{\cite[Definition 4.3]{bubenik-categorification-persistent-homology}}]\label{def:tame-filtration}
        Given a filtration $F$ of $\cat C_m$, we say that $a\in\R$ is a \textit{critical value} of $F$ if for all open interval $I$ containing $a$, there exist $u\leq v \in I$ such that $F_u^v$ is not an isomorphism.
        We call $F$ \textit{tame} if it has a finite number of critical values.
    \end{definition}
    In this setting, Lemma 4.4 of~\cite{bubenik-categorification-persistent-homology} implies that for sufficiently large $u$, all monomorphisms $F_u^v$ are isomorphisms.
    As a result, a tame filtration admits a colimit $F_\infty := F_{a+1}$ where $a$ is the largest critical value of $F$ (see~\cite{leinster-basic-category-theory} for the definition of a colimit).
    This colimit defines, for every $u\in \R$, a unique monomorphism $F_u^\infty$ from $F_u$ to $F_\infty$, such that for every $v\geq u$ we have $F_u^\infty = F_v^\infty \circ F_u^v$.
    It is then possible to define the ``filtering function'' associated to $F$:
    \begin{align*}
        f_F :\  \Phi F_\infty &\to \R \cup \{-\infty,+\infty\}\\
        x\ \ &\mapsto \inf\{u \ /\ x \in \im(\Phi F_u^\infty) \}
    \end{align*}
    }
    \add{
    Intuitively, a tame filtration $F$ can be viewed as a pair $(\Phi F_\infty, f_F)$, similarly to how a weighted graph filtration is described as a pair $(G,f)$.
    Tame filtrations are commonly used in practice for discrete objects such as graphs or simplicial complexes.
    Moreover, they are easier to study because their associated steady and ranging persistence diagrams are finite.
    From here, let $\F$ be a feature of the category $(\cat C_m, \Phi)$.
    \begin{lemma}\label{lem:tame-implies-finite-persistence}
        If $F$ is a tame filtration, \addd{then $Dgm(\sigma^\F_F)$ (resp. $Dgm(\rho^\F_F)$)} is finite i.e. there is a finite number of cornerpoints outside the diagonal $y=x$ and they have finite multiplicity.
    \end{lemma}
    See~\proofref{lem:tame-implies-finite-persistence} in the Appendix.
    }
    \smallskip
    
    \add{
    Following~\cite{bergomi-rank-based-persistence}, we get the following Lemma:
    \begin{lemma}[representation theorem~{\cite[Proposition 10]{bergomi-rank-based-persistence}}]\label{lem:radar}
        Given a tame filtration $F$ and $p = \sigma^\F_F$ (resp. $p = \rho^\F_F$) its associated persistence function, we have:
        $$ p(\bar u \leq \bar v) = \sum_{u<\bar u,\ v>\bar v} \mu(u,v) $$
        where $\mu$ is the multiplicity function associated to $p$ (see~\Cref{def:persistence-multiplicity}).
        \addd{The finiteness of the right side of the formula is ensured by~\Cref{lem:tame-implies-finite-persistence}.}
    \end{lemma}
    
    We restrict here the study of balanced and stable ip-generators to tame filtrations.
    An ip-generator $\pi$ is said to be \textit{tame-stable} if the stability property (\Cref{def:stability-ip-generators}) is satisfied for every tame filtration.
    Similarly, $\pi$ is \textit{tame-balanced} if the balanced property (\Cref{def:balanced-generator}) is satisfied for every tame filtration.
    In particular, stability implies tame-stability and balancedness implies tame-balancedness.
    We first show the following proposition:
    }
    \add{
    \begin{proposition}\label{prop:tame-stable-implies-tame-balanced}
      \addd{If $\sigma^\F$ (resp. $\rho^\F$) is tame-stable}, then it is tame-balanced.
    \end{proposition}
    See~\proofref{prop:tame-stable-implies-tame-balanced} in the Appendix.
    
    \smallskip
    
    However, the converse is less straightforward.
    It has been established in the context of weighted graphs in~\cite{bergomi-steady-ranging} and size functions in~\cite{amico-natural-pseudo-distance}.
    Here, we provide a sufficient condition on the category $(\cat C_m,\Phi)$ that allows to reproduce the arguments used in the proof of Theorem 29 in~\cite{amico-natural-pseudo-distance}, thus ensuring that the converse statement holds:
    \begin{definition}[triangle condition]\label{def:triangle-condition}
        A finitely concrete mono category $(\cat C_m,\Phi)$ satisfies the \textit{triangle condition} if, for any $f,g \in \morph(\cat C_m)$ and $\chi \in \morph(\cat{FinSet}_m)$, the equality $\Phi f = (\Phi g)\circ \chi$ implies the existence of a monomorphism $\phi$ such that $f = g \circ \phi$. This condition can be expressed using the following commutative diagrams:
        \begin{center}
\begin{tikzcd}
\Phi F \arrow[rd, "\Phi f"'] \arrow[rr, "\chi"] &        & \Phi G \arrow[ld, "\Phi g"] & \implies \quad \exists \phi & F \arrow[rd, "f"'] \arrow[rr, "\phi"] &   & G \arrow[ld, "g"] \\
                                                & \Phi H &                             &                             &                                       & H &                  
\end{tikzcd}
        \end{center}
    \end{definition}
    
    \begin{lemma}\label{lem:interleaved-equiv}
      Suppose that $(\cat C_m,\Phi)$ satisfies the triangle condition and let $F$ and $G$ be two tame filtrations.
      Then the two following properties are equivalent:
      \begin{itemize}
          \item $F$ and $G$ are $\epsilon$-interleaved.
          \item $\exists \phi : F_\infty \to G_\infty$ an isomorphism such that $||f_F - f_G \circ \Phi\phi||_\infty \leq \epsilon$.
      \end{itemize} 
    \end{lemma}
    See~\proofref{lem:interleaved-equiv} in the Appendix.
    }
    \smallskip
    
    \add{
    Finally, assuming the triangle condition, we show the converse of~\Cref{prop:tame-stable-implies-tame-balanced}:
    \begin{proposition}\label{prop:tame-balanced-implies-tame-stable-condition}
      Suppose that $(\cat C_m,\Phi)$ satisfies the triangle condition.
      \addd{If $\sigma^\F$ (resp. $\rho^\F$)} is tame-balanced, then it is tame-stable.
    \end{proposition}
    \begin{proof}
      Let $\pi$ be either $\sigma^\F$ or $\rho^\F$.
      This result is proved by adjusting the proof of Theorem 29 in~\cite{amico-natural-pseudo-distance} as follows:
      \begin{itemize}
          \item size pairs such as $(\mathcal M, \phi)$ should be replaced by tame filtrations such as $F$, with associated pair $(\Phi F_\infty, f_F)$.
          \item Reduced size function $l^\ast_{(\mathcal M, \phi)}$ should be replaced by persistence function $\pi(F)$.
          \item Theorem 8 corresponds to~\Cref{lem:radar}.
          \item Proposition 10 is equivalent to the assumption that $\pi$ is tame-balanced.
          This is a direct consequence of~\Cref{lem:interleaved-equiv}.
          Note that ``$f:\mathcal M \to \mathcal N$ is a homeomorphism'' should be replaced by ``$\phi :F_\infty \to G_\infty$ is an isomorphism'' and ``$\max_{P\in\mathcal M}|\phi(P)-\psi(f(P))|\leq h$'' by ``$||f_F - f_G \circ \Phi\phi||_\infty \leq \epsilon$''.
          \item Proposition 11 to 14 are consequences of~\Cref{def:persistence-function} and of the tameness of the considered filtrations (mainly~\Cref{lem:tame-implies-finite-persistence}).
      \end{itemize}
      Theorem 29 states that if $\pi$ is tame-balanced, we have $d_B(\pi(F), \pi(G)) \leq \inf_{\phi} ||f_F - f_G \circ \Phi\phi||_\infty$.
      \Cref{lem:interleaved-equiv} implies $\inf_{\phi} ||f_F - f_G \circ \Phi\phi||_\infty = \delta(F,G)$ so we get:
      \begin{center}
      $d_B(\pi(F), \pi(G)) \leq \delta(F,G)$
      \end{center}
    \end{proof}
  }
  \addd{In~\Cref{sect:hypergraph-categories}, based on the concept of hypergraph, two categories that satisfy and one that does not satisfy the triangle condition are presented.}

\section{Convex and Balanced Features}\label{sect:convex-balanced}
    The stability of persistence diagrams is one of the main advantages of persistent homology.
    Thus, it is natural to wonder whether the stability of persistence diagrams transfers to steady and ranging persistence.
    
    As discussed in the previous section, the stability and balancedness of ip-generators are closely related. 
    The authors of~\cite{bergomi-steady-ranging} showed that some features induce balanced steady and ranging generators, whereas other features induce generators that are not balanced.
    A feature is said to be \emph{steady-balanced} (resp. \emph{ranging-balanced}) if its induced steady (resp. ranging) generator is balanced (see~\Cref{def:steady-ranging-generators,def:balanced-generator}).
    In the aforementioned paper, they defined specific classes of steady-balanced and ranging-balanced features (see~\Cref{prop:graph-simple-feature}).
    
    In this section, we investigate the balancedness of steady and ranging persistence and establish a complete characterization of steady-balanced and ranging-balanced features (see our main result~\Cref{coro:convex-balanced-equivalence}).
    Precisely, we define the notion of convex feature and compare it to the previous notions of monotone and simple features.
    We then show that convex features are equivalent to steady-balanced and ranging-balanced features.

        \begin{definition}[convex feature]\label{def:convex-feature}
            A feature $\F$ is said to be \emph{convex} if
            $$(A,X)\in\F\ \tx{and}\ \iota'\iota(A,X)\in\F \implies \iota(A,X)\in\F$$
            where $\iota$ and $\iota'$ are composable monomorphisms of $\morph(\cat C_m)$, and $\iota'\iota$ is the notation for $\iota'\circ\iota$.
        \end{definition}

        \begin{proposition}\label{prop:convex-feature-charac}
            A feature $\F$ is convex if and only if $S^\F = R^\F$.
        \end{proposition}
        See~\proofref{prop:convex-feature-charac} in the Appendix.

        \subsection{Examples and constructions of convex features}
        In this section, we present different classes of features and expose their relation to convex features.
        In particular, \textit{monotone} and \textit{simple} features (that were previously defined for graphs, see~\Cref{def:graph-monotone-feature,def:graph-simple-feature}) are redefined using the proposed paradigm, to extend their scope to other categories.
        We also introduce two useful classes of features, namely the features that are \textit{right-continued} and \textit{left-continued}.

        \Cref{prop:convex-feature-examples} shows that the features that are monotone or right-continued (resp. left-continued) are convex.
        We then prove that being a simple feature is equivalent to being a convex feature (see~\Cref{prop:convex-feature-examples,prop:simple-convex-equivalence}).

        \begin{definition}[monotone feature (generalization of~\Cref{def:graph-monotone-feature})]\label{def:monotone-feature}
            A feature $\F$ is considered \emph{monotone} if it satisfies the following properties:
            \begin{enumerate}
                \item $\iota(A,X) \in \F \implies (A,X) \in \F$ 
                \item $B\subseteq A\ \tx{and}\ (A,X)\in\F \implies (B,X)\in\F$.
            \end{enumerate}
        \end{definition}
        \begin{definition}[simple feature (generalization of~\Cref{def:graph-simple-feature})]\label{def:simple-feature}
            A feature is considered \emph{simple} if it satisfies the following property:
            
            for every $X_1\injto{\iota}X_2\injto{\iota'}X_3$, we have
            $$\iota\iota'(A,X_1)\in\F \ \tx{and} \ \iota(A,X_1)\notin\F \implies (A,X_1)\notin \F$$
        \end{definition}

        \begin{definition}[continued feature]\label{def:closed-monomorphism}
            A feature $\F$ is considered \emph{right-continued} if $$(A,X)\in\F\implies \iota(A,X)\in\F$$
            Similarly, a feature $\F$ is considered \emph{left-continued} if $$\iota(A,X)\in\F\implies(A,X)\in\F$$
        \end{definition}

        \begin{proposition}\label{prop:convex-feature-examples}
            The following classes of features are convex:
            \add{
            \begin{enumerate}
                \item right-continued;
                \item left-continued;
                \item monotone.
            \end{enumerate}
            }
        \end{proposition}
        \begin{proof}

Using~\Cref{def:convex-feature,def:closed-monomorphism} and the fact that $(\cal{A} \Rightarrow \cal{C}) \Rightarrow ((\cal{A} \wedge \cal{B}) \Rightarrow \cal{C})$:
\begin{enumerate}
    \item right-continued features are convex;
    \item left-continued features are convex because in this case, $\iota'\iota(A,X)\in\F$ implies $\iota(A,X)\in\F$;
    \item monotone features are convex because the first condition of monotone features is to be left-continued (see~\Cref{def:monotone-feature}).
\end{enumerate}

        \end{proof}

        \begin{proposition}\label{prop:simple-convex-equivalence}
            A feature is simple if and only if it is convex.
        \end{proposition}
        See~\proofref{prop:simple-convex-equivalence} in the Appendix.\\

        Finally, we define the \textit{maximal} and \textit{minimal} versions of a feature and provide sufficient conditions for them to be convex:
        \begin{definition}[maximal and minimal version of a feature]\label{def:maximal-minimal-feature}
            The maximal and minimal versions of a feature $\F$ are defined as:
            \begin{align*}
                M\F &:= \left\{ (A,X)\in\F\ /\ \tx{ if }\ A\subseteq B\ \tx{ and }\ (B,X)\in\F \ \tx{ then }\  A=B \right\}\\
                m\F &:= \left\{ (A,X)\in\F\ /\ \tx{ if }\ A\supseteq B\ \tx{ and }\ (B,X)\in\F \ \tx{ then }\  A=B \right\}
            \end{align*}
        \end{definition}

        \begin{proposition}\label{prop:max-min-ric-feature}
            If $\F$ is right-continued, then $M\F$ and $m\F$ are convex.
        \end{proposition}
        See~\proofref{prop:max-min-ric-feature} in the Appendix.

    \subsection{Balanced features are equivalent to convex features}
        The main result is presented in this section: convex features are exactly the features that induce balanced ip-generators (see~\Cref{coro:convex-balanced-equivalence}).

        We first prove that convex features are steady-balanced and ranging-balanced.
        This statement can also be obtained for the category of graphs using~\Cref{prop:simple-convex-equivalence} and because simple features induce balanced $\sigma^\F$ and $\rho^\F$ (see~\cite{bergomi-exploring-graph-persistence}).
        In this study, we provide a general proof.
        \begin{proposition}\label{thm:convex-implies-balanced}
            If $\F$ is a convex feature, then $\sigma^\F = \rho^\F$ is a balanced ip-generator.
        \end{proposition}
        See~\proofref{thm:convex-implies-balanced} in the Appendix.\\

        It turns out that the features that induce balanced steady-generator or balanced ranging-generator are exactly the convex features.
        First, we prove that a feature that is not convex induces a steady-generator that is not balanced, i.e., that steady-balanced features are convex.
        \begin{proposition}\label{thm:steady-balanced-implies-convex}
            If a feature $\F$ is not convex, then it is not steady-balanced.
        \end{proposition}
        See~\proofref{thm:steady-balanced-implies-convex} in the Appendix.
        
        \smallskip

        With a similar proof (yet a bit more technical), we show that ranging-balanced features are convex.
        \begin{proposition}\label{thm:ranging-balanced-implies-convex}
            If a feature $\F$ is not convex, then it is not ranging-balanced.
        \end{proposition}
        See~\proofref{thm:ranging-balanced-implies-convex} in the Appendix.
        
        \smallskip
        
        All in all, we obtain our main result, which is the following characterization of convex and balanced features:
        \begin{theorem}\label{coro:convex-balanced-equivalence}
            Given a feature $\F$ of the category $\cat C_m$, the following properties are equivalent:
            \begin{enumerate}
                \item $\F$ is convex;
                \item $\sigma^\F=\rho^\F$;
                \item $\sigma^\F$ is balanced;
                \item $\rho^\F$ is balanced.
            \end{enumerate}
        \end{theorem}
        \begin{proof}
            By~\Cref{prop:convex-feature-charac,thm:convex-implies-balanced,thm:steady-balanced-implies-convex,thm:ranging-balanced-implies-convex}.
        \end{proof}
        \add{
        \begin{remark}
            Note that being balanced implies being tame-balanced.
            Moreover, the filtrations used in the proofs of~\Cref{thm:steady-balanced-implies-convex,thm:ranging-balanced-implies-convex} are tame.
            Therefore, if a feature is not convex, then its steady and ranging generators are not tame-balanced.
            As a result,~\Cref{coro:convex-balanced-equivalence} remains valid when replacing ``balanced'' with ``tame-balanced''.
            In particular, $\sigma^\F$ (or $\rho^\F$) is balanced if and only if it is tame-balanced.
        \end{remark}
        }
        \add{An advantage of this characterization is that the convex property is defined on the feature itself and is simpler to handle than the balanced property, which requires to consider $\epsilon$-interleaved filtrations and their persistence functions.}
        In addition,~\Cref{coro:convex-balanced-equivalence} implies that the stability of steady and ranging persistence is entangled: $\sigma^\F$ is balanced if and only if $\rho^\F$ is balanced (i.e., being steady-balanced is equivalent to being ranging-balanced).
        
        \smallskip
        
        \begin{remark}
            Let us remind the assumptions made for~\Cref{coro:convex-balanced-equivalence} to be true:
            \begin{enumerate}
                \item $(\cat C_m, \Phi)$ is a \emph{finitely concrete mono category}.
                \item $(\cat C_m, \Phi)$ has an \emph{empty-like object} $\hat\emptyset$.
            \end{enumerate}
        \end{remark}

\section{Steady and Ranging persistence for Hypergraphs}\label{sect:examples-constructions-hypergraphs}
    This section provides specific examples of steady and ranging persistence with the new formalism introduced in~\Cref{sect:extending-steady-ranging}.
    
    In~\cite{liu-persistent-spectral-hypergraph-learning}, the authors study the persistence of hypergraph attributes in filtrations derived from proteins.
    Inspired by this work, we present results and examples obtained using steady and ranging persistence for hypergraphs.
    We introduce different hypergraph features to give some insight about~\Cref{coro:convex-balanced-equivalence} concerning the convexity and stability of features.

    \subsection{Hypergraph categories}\label{sect:hypergraph-categories}
    Hypergraphs are generalization of graphs, where edges can connect more than two nodes.
    In this section, we propose a formalism for hypergraphs to fit our categorical framework.
    The definitions are similar to those used in~\cite{ouvrard-hypergraphs,dorfler-category-hypergraphs}.
    
    \begin{definition}[hypergraph]\label{def:hypergraph}
        A \emph{hypergraph} $H=(V,E,h)$ consists of a finite set of vertices $V$, a finite set of \textit{hyperedges} $E$, and an \textit{incidence function} $h:E\to\mathcal{P}(V)\bs\emptyset$ that associates each hyperedge with its corresponding set of vertices.
        \addd{In particular, two hyperedges of a hypergraph can have the same underlying set of vertices.}
    \end{definition}
    \begin{definition}[hypergraph morphism]\label{def:hypergraph-morphism}
        A \emph{hypergraph morphism} from $H=(V,E,h)$ to $H'=(V',E',h')$ is a function  $f: V\sqcup E \to V'\sqcup E'$ such that
        \begin{enumerate}
            \item $f(V)\subseteq V'$
            \item $f(E)\subseteq E'$
            \item for all $(v,e)\in V\times E$, we have $v\in h(e) \implies f(v) \in h'(f(e))$.
        \end{enumerate}
    \end{definition}    
    From now on, we will abuse the notation and mix a hyperedge with its underlying set.
    For instance, if $v$ is a vertex and $e$ is a hyperedge, we will write $v \in e$, $|e|$, and $e \cap e'$ instead of $v\in h(e)$, $|h(e)|$, and $h(e) \cap h(e')$, respectively.
    \smallskip
    
    Hypergraphs enter the scope of the aforementioned generalization of steady and ranging persistence.
    Indeed, it is possible to define a hypergraph category that satisfies the three assumptions of~\Cref{sect:redefining-steady-ranging}.
    \begin{definition}[hypergraph category]\label{def:hypergraph-category}
        $\cat{Hgph}$ is the category whose objects are hypergraphs and whose morphisms are hypergraph morphisms.
        $\cat{Hgph}_m$ is the category $\cat{Hgph}$ restricted to monomorphisms.
    \end{definition}

    \begin{proposition}[]\label{prop:hypergraph-category-satisfy-assumptions}
        $\cat{Hgph}_m$ is a finitely concrete mono category with an empty-like element, i.e.\ it satisfies the assumptions of~\Cref{sect:redefining-steady-ranging}.
    \end{proposition}
    \begin{proof}
We define the following functor:
\begin{align*}
    \Phi :\quad \obj(\cat{Hgph}_m)\quad &\to\quad \obj(\cat{FinSet}_m)\\
         (V,E,h)\qquad &\mapsto\qquad V \sqcup E\\
    \Phi :\ \morph(\cat{Hgph}_m)\ &\to\ \morph(\cat{FinSet}_m)\\
        \big( (V,E,h)\injto{\iota}(V',E',h') \big) &\mapsto\ \Phi\iota\ :=\ \iota : V\sqcup E \to V'\sqcup E'
\end{align*}

First, $\Phi$ is a functor from $\cat{Hgph}_m$ to $\cat{FinSet}_m$.
Indeed, every monomorphism $\iota$ of $\cat{Hgph}$ is injective (a proof is provided in the appendix, see~{\hyperlink{mono-implies-injective}{Proof of injectivity of monomorphisms in $\cat{Hgph}$}}) therefore $\Phi\iota = \iota$ is also injective.

Second, we prove that $\Phi$ is faithful:
let $H\injto{\iota}H'$ and $H\injto{\iota'}H'$ be two monomorphisms such that $\Phi\iota = \Phi\iota'$. 
We have $\iota = \Phi\iota = \Phi\iota'(x) = \iota'$.
Hence, $\Phi$ is faithful.

Third, let $H_\emptyset = (\emptyset, \emptyset, \vec\emptyset)$ be the empty hypergraph, with no node.
For every hypergraph $H$ there exists one only one monomorphism from $H_\emptyset$ to $H$ (which consists of the empty function $\emptyset \to V\sqcup E$).
Moreover, $\Phi H_\emptyset =\emptyset$ so $H_\emptyset$ is an empty-like object.

    \end{proof}

    It is possible to restrain the category $\cat{Hgph}_m$ to allow only certain monomorphisms between consecutive hypergraphs.
    This can be useful because the behavior of features depends on the underlying category, and some practical hypergraph filtrations can fall into those restricted categories.

    Thus, two alternative hypergraph categories are defined:
    \begin{definition}[hypergraph alternative categories]\label{def:hypergraph-alternative-categories}~

        \begin{enumerate}
            \item $\cat{Hgph}^\leq_m$ is the category $\cat{Hgph}_m$ restricted to monomorphisms $f$ such that\\
            $f(u)\in f(e)\implies u\in e$ (equivalently $u\notin e\implies f(u)\notin f(e)$).
            \item $\cat{Hgph}^=_m$ is the category $\cat{Hgph}_m$ restricted to monomorphisms $f$ such that\\
            $|e| = |f(e)|$ (preserving the size of hyperedges).
        \end{enumerate}
        An illustration of the different types of hypergraph monomorphisms is shown in~\Cref{fig:hypergraph-filtration-types}.
    \end{definition}
    
    \smallskip
    
    \Cref{def:hypergraph-alternative-categories} implies that  $\morph(\cat{Hgph}^=_m)\subseteq\morph(\cat{Hgph}^\leq_m)\subseteq\morph(\cat{Hgph}_m)$.
    \addd{Indeed, let us prove the first inclusion: suppose that $|e| = |f(e)|$ and let $f(u)\in f(e)$, we show that $u\in e$.
    Hypergraph monomorphisms are injective for vertices and edges (see~{\hyperlink{mono-implies-injective}{Proof of injectivity of monomorphisms in $\cat{Hgph}$}}), so $|\{f(v), v\in e\}|=|e|$, which implies $|\{f(v), v\in e\}|=|f(e)|$.
    As $\{f(v), v\in e\}\subseteq f(e)$ (by condition 3 of~\Cref{def:hypergraph-morphism}), we obtain $\{f(v), v\in e\} = f(e)$.
    Hence, there exists $v\in e$ such that $f(u) = f(v)$, which gives $u=v\in e$ by injectivity of $f$.}
    
    \smallskip
    %
    These inclusions imply that the convex features of $\cat{Hgph}^\leq_m$ are convex in $\cat{Hgph}^=_m$ and the convex features of $\cat{Hgph}_m$ are convex in $\cat{Hgph}^\leq_m$.
    However, the converse is not necessarily true, as shown by the examples in~\Cref{sect:hypergraphs-features}.
    
    \begin{figure}
          \centering
        \includegraphics[width = 0.43\textwidth]{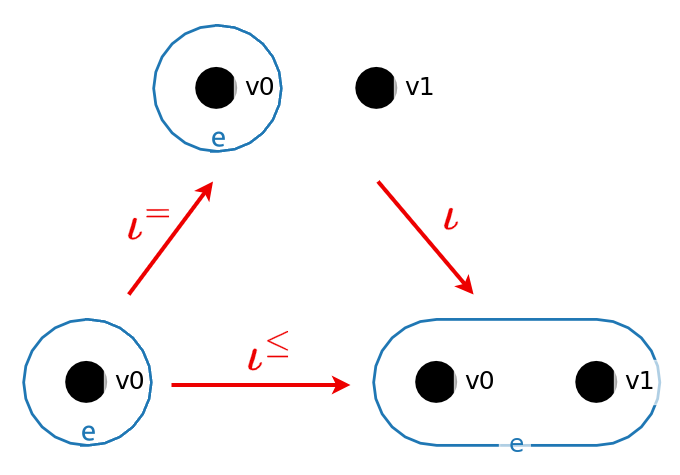}
        \includegraphics[width = 0.27\textwidth]{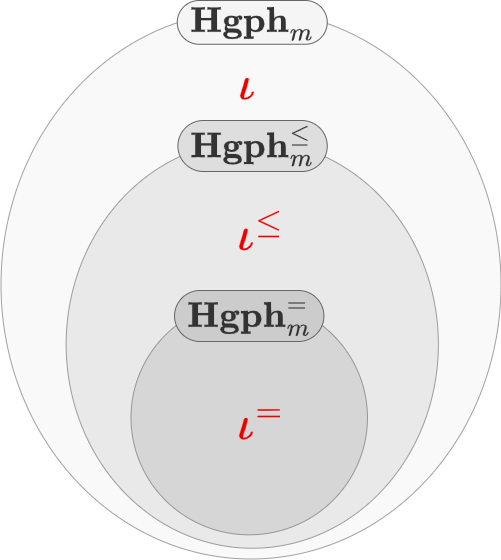}
          \caption{On the left: three hypergraphs (nodes in black and hyperedges in blue) and three monomorphisms (in red) of the hypergraph category.
          $\iota$ is a morphism of $\cat{Hgph}_m$ but not of $\cat{Hgph}^\leq_m$ and $\cat{Hgph}^=_m$.
          $\iota^\leq$ is a morphism of $\cat{Hgph}_m$ and $\cat{Hgph}^\leq_m$ but not of $\cat{Hgph}^=_m$.
          $\iota^=$ is a morphism of $\cat{Hgph}_m$, $\cat{Hgph}^\leq_m$ and $\cat{Hgph}^=_m$.}
          \label{fig:hypergraph-filtration-types}
    \end{figure}
    \add{
      Here, we present how the results from~\Cref{sect:stability-balancedness} apply to hypergraph categories:
      \begin{proposition}\label{prop:triangle-condition-hypergraph}
      $(\cat{Hgph}^\leq_m,\Phi)$ and $(\cat{Hgph}^=_m,\Phi)$ satisfy the triangle condition (see~\Cref{def:triangle-condition}).
      As a result, by~\Cref{prop:tame-stable-implies-tame-balanced,prop:tame-balanced-implies-tame-stable-condition,coro:convex-balanced-equivalence}, the following properties are equivalent:
      \begin{itemize}
          \item $\F$ is a convex feature of $\cat{Hgph}^\leq_m$ (or $\cat{Hgph}^=_m$).
          \item $\sigma^\F = \rho^\F$ is tame-balanced in $\cat{Hgph}^\leq_m$ (or $\cat{Hgph}^=_m$).
          \item $\sigma^\F = \rho^\F$ is tame-stable in $\cat{Hgph}^\leq_m$ (or $\cat{Hgph}^=_m$).
      \end{itemize}
      \end{proposition}
      See~\proofref{prop:triangle-condition-hypergraph} in the Appendix.
    }
    
    \addd{However, $(\cat{Hgph}_m,\Phi)$ does not satisfy the triangle condition, see~\Cref{fig:hgph-triangle}.}
    \begin{figure}[!htb]
        \centering
        \includegraphics[width = 0.58\textwidth]{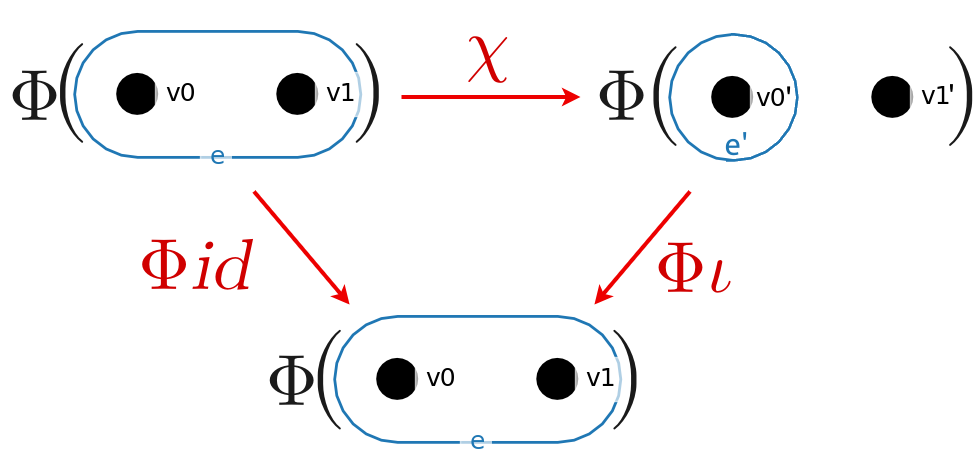}
          \caption{\addd{Images of three hypergraphs under the functor $\Phi$, together with the monomorphisms $\Phi\id$, $\Phi\iota$, and $\chi$ between the corresponding finite sets.
          The morphism $\iota$ is a hypergraph monomorphism in $\cat{Hgph}_m$. Both $\iota$ and its image $\Phi\iota\in \morph(\cat{FinSet}_m)$ send $e'$, $v_0'$, and $v_1'$ to $e$, $v_0$, and $v_1$, respectively.
          Conversely, $\chi\in \morph(\cat{FinSet}_m)$ sends $e$, $v_0$ and $v_1$ to $e'$, $v_0'$ and $v_1'$.
          The diagram commutes, but $\chi$ cannot be obtained as the image under $\Phi$ of any hypergraph monomorphism (this would contradict condition 3 of~\Cref{def:hypergraph-morphism}).
          Thus, $(\cat{Hgph}_m,\Phi)$ does not satisfy the triangle condition.}}
          \label{fig:hgph-triangle}
    \end{figure}
    
  \subsection{Hypergraph features}\label{sect:hypergraphs-features}
    In this section, we present three different hypergraph features: the \textit{hub}, the \textit{exclusivity}, 
    and the \textit{max-originality} features.
    We discuss their convexity using proofs or counterexamples.
    \add{Note that the invariance of the proposed features under hypergraph isomorphism (see~\Cref{def:feature-concrete}) is not explicitly proven here, as it is inherent by construction. Indeed, the definitions of these features depend exclusively on the underlying hypergraph structure, which is preserved by isomorphism.
    }
    
    To shorten notation, we will omit the incidence function and write $H=(V,E)$ instead of $H=(V,E,h)$.
    Given a hyperedge $e$ of hypergraph $H$, $N(e)$ denotes the set of other hyperedges that share at least one vertex with $e$.
    $N(e)$ is called the set of neighbors of $e$.
    \begin{enumerate}
        \item \emph{The hub feature}. \add{Inspired by} the study of persistent hubs of graphs in~\cite{bergomi-steady-ranging},
            a hyperedge $e$ is said to be a \textit{hub} in hypergraph $H$ if it has more neighbors than its neighbors:
            $$e\ \tx{is a hub in $H$}\ \equi\ N(e)\neq\emptyset\ \tx{and}\ |N(e)|>max_{e'\in N(e)}|N(e')|$$
            The hub feature for hypergraphs is defined as follows:
            $$ \F^h = \left\{ (A,H), \tx{with}\ H=(V,E)\ \tx{such that}\ A=\{e\}\subseteq E\ \tx{and}\ e\ \tx{is a hub in $H$} \right\}$$

            In~\cite{bergomi-steady-ranging}, the hub feature for vertices in graphs was proved to induce unbalanced steady and ranging generators.
            Similarly, the present hub feature (for hyperedges in hypergraphs) is not convex in $\cat{Hgph}^=_m$ (see~\Cref{fig:hyper-filtration-2}); thus, it induces unbalanced generators.
            \begin{figure}[!htb]
                \centering
                \includegraphics[width = 0.65\textwidth]{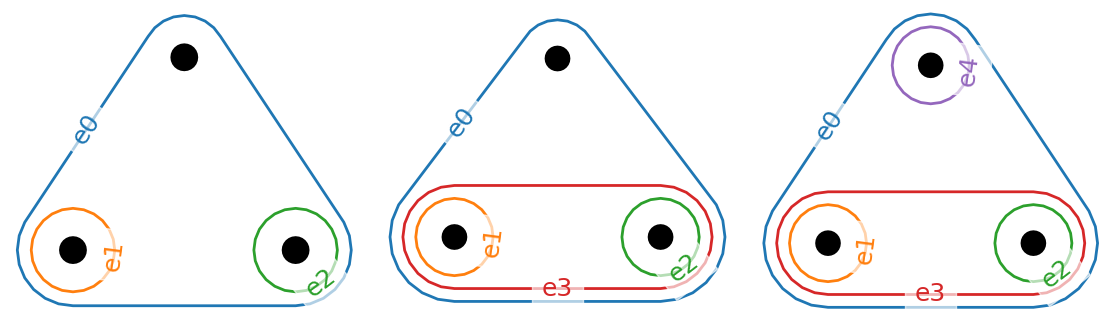}
                  \caption{A hypergraph filtration $H_0 \injto H_1 \injto H_2$ in $\cat{Hgph}^=_m$.
                  $e_0$ is a hub in $H_0$ and $H_2$ but not in $H_1$ so the feature $\F^h$ is not convex in $\cat{Hgph}^=_m$.
                  Hence, it is not convex in $\cat{Hgph}^\leq_m$ and $\cat{Hgph}_m$}
                  \label{fig:hyper-filtration-2}
            \end{figure}

        \item \emph{The exclusivity feature}.
            A hyperedge $e$ is said to \textit{have an exclusivity} in hypergraph $H$ if and only if it possesses a node that does not belong to any other hyperedge:
            $$e\ \tx{has an exclusivity in $(V,E)$}\ \equi\ \exists v\in e\ /\ \forall e'\in E\bs\{e\},\ v\notin e'$$
            The exclusivity feature is defined as follows:
            $$ \F^x = \left\{ (A,H), \tx{with}\ H=(V,E)\ \tx{such that}\ A=\{e\}\subseteq E \ \tx{and}\ e\ \tx{has an exclusivity} \right\}$$

            \begin{proposition}\label{prop:convex-exclusivity}
                $\F^x$ is convex in $\cat{Hgph}^=_m$.
            \end{proposition}
            \begin{proof}
We show that $\F^x$ is left-continued (and hence convex, see~\Cref{prop:convex-feature-examples}) in $\cat{Hgph}^=_m$:
consider $(V,E)\injto{\iota}(V',E')$ and suppose $\iota(\{e\}, (V,E))\in \F^x$: there exists $v'\in \iota(e) \subseteq V'$ such that $\forall e'\in E'\bs\{\iota(e)\},\ v'\notin e'$.

As we consider $\cat{Hgph}^=_m$, we have that $\iota$ is an injective hypergraph morphism that preserves the size of the hyperedges.
As a consequence, it acts as a bijection between the set of $e$ and the set of $\iota(e)$:
thus there exists $v\in e$ such that $v' = \iota(v)$.
\smallskip

Let $\tilde e \in E\bs\{e\}$.
By injectivity of $\iota$, we have that $\iota(\tilde e) \neq \iota (e)$.
As $\iota(e)$ has an exclusivity (which is $v'$), we get that $v' =\iota(v)\notin \iota(\tilde e)$.
\add{Therefore, $v\notin \tilde e$ (by the third property of hypergraph morphism)}.

Thus, $e$ has an exclusivity in $(V,E)$ (which is $v$), hence, $(\{e\}, (V,E))\in \F^x$.
As a conclusion, $\F^x$ is left-continued.
            \end{proof}
            However, $\F^x$ is not convex in $\cat{Hgph}^\leq_m$ (see~\Cref{fig:hyper-filtration-1}).


        \item \emph{The max-originality feature}. The authors of~\cite{shibayama-originality} defined the \textit{originality} of a node in a directed graph by observing the relations between the ``inward'' neighborhood and the ``outward'' neighborhood.
            Inspired by this definition, we define the \textit{max}-originality $O_H(e)$ of a hyperedge $e$ in $H$:
            $$ O_H(e) = 1-\frac{1}{|e|}\max_{e'\in N(e)}|e \cap e'| \qquad \tx{or}\quad O_H(e)=1 \quad \tx{if $N(e)=\emptyset$}$$
            The max-originality feature is then defined as follows:
            $$ \F^{O} = \left\{ (A,H), \tx{with}\ H=(V,E)\ \tx{such that}\ A=\{e\}\subseteq E\ \tx{and}\ O_H(e)>\half \right\}$$
            Intuitively, a hyperedge $e$ is represented in $\F^{O}$ if no other hyperedge contains more than half of the nodes in $e$.

            \begin{proposition}\label{prop:convex-max-original}
                $\F^{O}$ is convex in $\cat{Hgph}^=_m$.
            \end{proposition}
            \begin{proof}
We show that $\F^{O}$ is left-continued (and hence convex, see~\Cref{prop:convex-feature-examples}) in $\cat{Hgph}^=_m$.
Let $H\injto{\iota}H'$ and suppose that $\iota({e},H)\in \F^{O}$ (in particular, $O_{H'}(\iota(e))>1/2$).
Let $\tilde e\in N(e)$ and $v\in e\cap \tilde e$.

$\iota$ is a hypergraph morphism, so we have $\iota(v) \in \iota(e)\cap\iota(\tilde e)$, thus, $\iota(\tilde e)\in N(\iota(e))$.
$\iota$ is injective so we obtain that $|e\cap \tilde e|\leq |\iota(e)\cap\iota(\tilde e)|$, hence $\max_{e'\in N(e)}|e\cap e'|\leq \max_{e'\in N(e)}|\iota(e)\cap \iota(e')| \leq \max_{e''\in N(\iota(e))}|\iota(e)\cap e''|$.
Moreover, $\iota$ preserves the size of $e$ so $|e| = |\iota(e)|$.
All in all we get:
\begin{align*}
    O_H(e) &= 1-\frac{1}{|e|}\max_{e'\in N(e)}|e \cap e'|\\
    &\geq\ 1-\frac{1}{|\iota(e)|}\max_{e''\in N(\iota(e))}|\iota(e) \cap e''|\ =\ O_{H'}(\iota(e))\\
    &> 1/2
\end{align*}
As a conclusion, $({e},H)\in \F^{O}$.

            \end{proof}
            However, $\F^{O}$ is not convex in $\cat{Hgph}^\leq_m$ (see~\Cref{fig:hyper-filtration-1}).
            \begin{figure}[!htb]
              \centering
                
              \begin{tabular}{ccc}
              \includegraphics[height = 0.11\textheight]{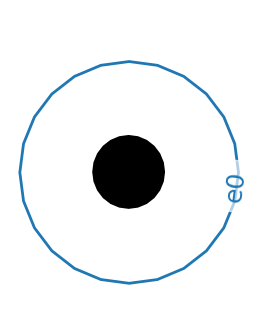} &
              \includegraphics[height = 0.11\textheight]{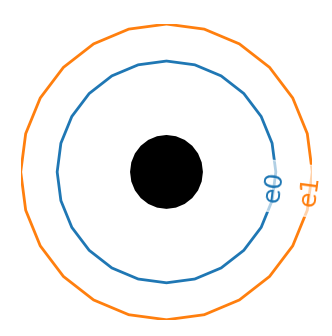} &
              \includegraphics[height = 0.11\textheight]{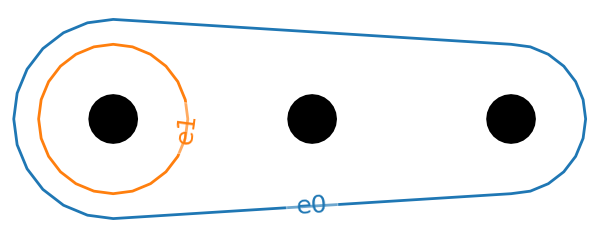}\\
              ($H_0$) & ($H_1$)& ($H_2$) 
              \end{tabular}
              \caption{A hypergraph filtration $H_0 \injto H_1 \injto H_2$ in $\cat{Hgph}^\leq_m$.
              $e_0$ has an exclusivity in $H_0$ and $H_2$, but not in $H_1$. This implies that $\F^x$ is not convex in $\cat{Hgph}^\leq_m$ (nor is it in $\cat{Hgph}_m$).
              Moreover, the max-originality values of $e_0$ are $O_{H_0}(e_0) = 1$, $O_{H_1}(e_0) = 0$ and $O_{H_2}(e_0) = 2/3$.
              Hence, the feature $\F^{O}$ is not convex in $\cat{Hgph}^\leq_m$ (nor is it in $\cat{Hgph}_m$).}
              \label{fig:hyper-filtration-1}
            \end{figure}
    \end{enumerate}

    \subsection{Experiments}

    In this section, we present some results of steady and ranging persistence obtained using the hypergraph features defined in the previous section.
    For these experiments, we build hypergraphs from the Hyperbard dataset~(\cite{hyperbard}), which consists of data about Shakespeare plays.
    For a given play, we build a hypergraph whose vertices are characters and whose hyperedges correspond to scenes.
    A vertex belongs to a hyperedge if the corresponding character appears in the corresponding scene.
    We call this hypergraph a \textit{scene-hypergraph}.

    We also consider the dual of the scene-hypergraph (namely the \textit{character-hypergraph}), whose vertices are scenes and whose hyperedges are characters.
    In the character-hypergraph, a hyperedge contains a vertex if the corresponding character appears in the corresponding scene.

    The filtration is chronological: the $i$-th scene appears at $t=i$, and a character appears simultaneously with the first scene it is in.
    This construction induces that the scene-hypergraph filtration is a filtration of $\cat{Hgph}^=_m$, whereas the character-hypergraph filtration is a filtration of $\cat{Hgph}^\leq_m$.

    Indeed, if $f:H\injto H'$ in a scene-hypergraph filtration, then $|f(e)|=|e|$ because a scene cannot gain a new character after occurring.
    Similarly, if $f:H\injto H'$ in a character-hypergraph filtration, and $u\notin e$ (i.e., character $e$ does not appear in scene $u$), then $f(u)\notin f(e)$ (i.e., character still does not appear in the scene).

    The results obtained from the scene-hypergraph and character-hypergraph filtrations derived from the play \textit{King Lear} are presented here.
    Other results from \textit{Romeo and Juliet} are shown in Appendix (\Cref{sect:appendix-results}).

    The experiments were coded in Python, using the HyperNetX library~(\cite{hypernetx}) to handle hypergraphs.
    Hypergraph graphical representations were also created using HyperNetX.
    Part of the persistence code was borrowed from the python package developed in~\cite{bergomi-steady-ranging}.
    The code for the experiments is publicly available in an online repository, see~\cite{ys-github-steady-ranging-persistence}.
    
    \begin{figure}[H]
        \centering
        \includegraphics[width = 0.7\textwidth]{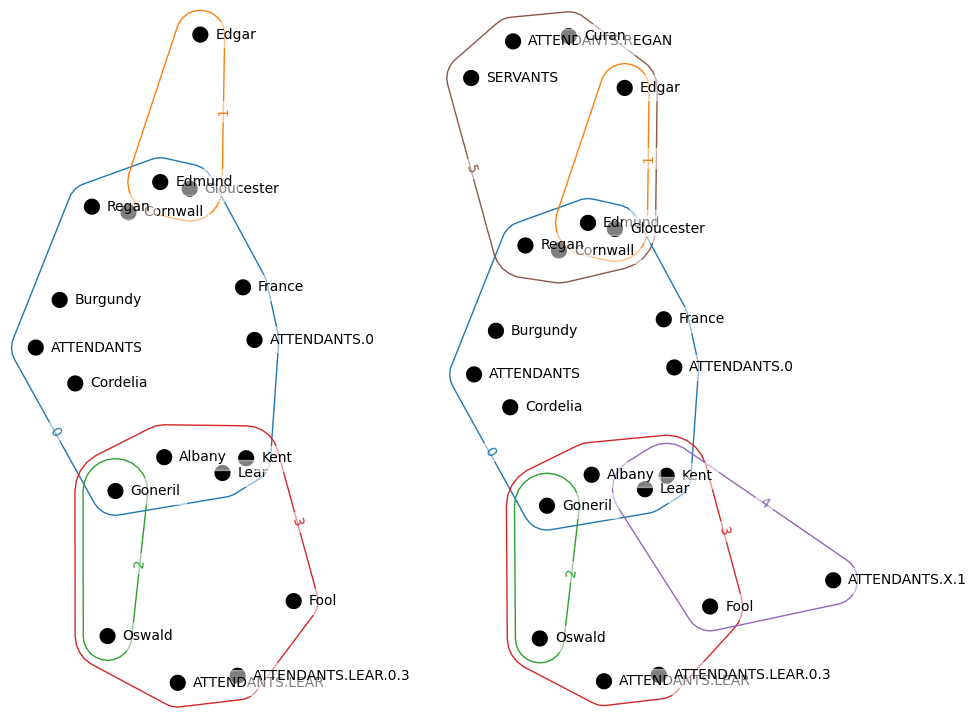}
          \caption{The scene-hypergraph filtration of \textit{King Lear} at $t=3$ and $t=5$. The play has 26 scenes, so the filtration starts at $t=0$, and is constant after $t=25$. The hypergraphs rapidly start being dense and unreadable.}
          \label{fig:filtration-scene}
    \end{figure}
    \paragraph{Results for the scene-hypergraph of \textit{King Lear}}    

    \begin{figure}[H]
        \caption{Steady persistence of the hub feature $\F^h$ for the scene-hypergraph filtration induced by \textit{King Lear}.
        The 0-th scene starts being a hub at $t=1$, and stops being a hub between $t=6$ and $t=16$ before becoming a hub again at $t=25$.
        This induces a difference between the steady and ranging diagrams, which illustrates that $\F^h$ is not convex in $\cat{Hgph}^=_m$ (see~\Cref{fig:ranging-hubs-scene-king_lear}).}
        \centering
        \includegraphics[width =0.7\textwidth]{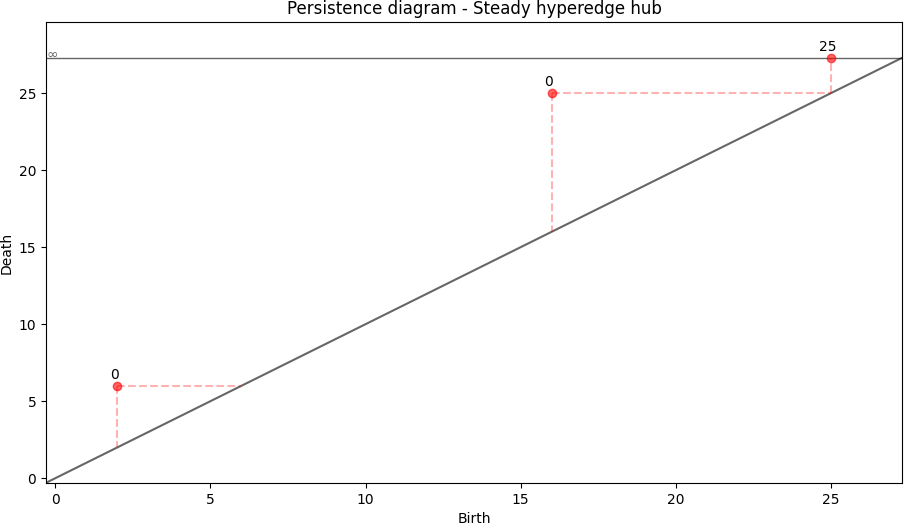}
        \label{fig:steady-hubs-scene-king_lear}
    \end{figure}

    \begin{figure}[H]
        \caption{Ranging persistence of the hub feature $\F^h$ for the scene-hypergraph filtration induced by \textit{King Lear}.
        The first and the last scene (0-th and 25-th scenes) are the only hubs that appear in the filtration.}

        \centering
        \includegraphics[width =0.7\textwidth]{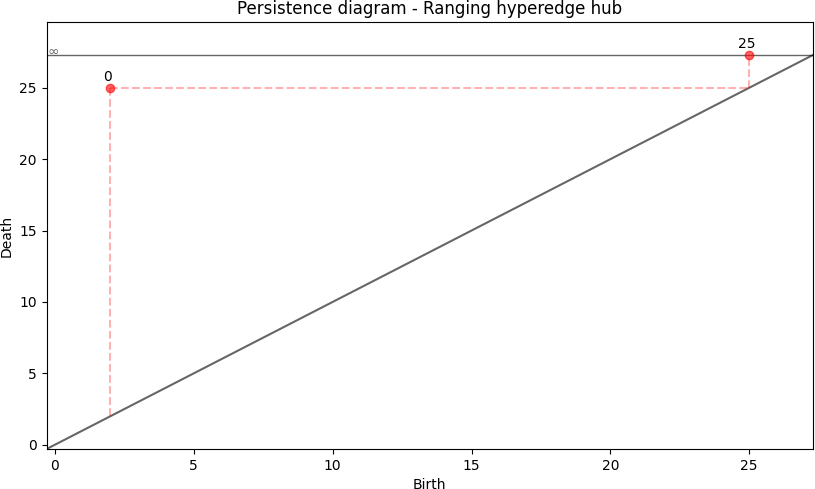}
        \label{fig:ranging-hubs-scene-king_lear}
    \end{figure}

    \begin{figure}[H]
        \caption{Persistence of $\F^x$ for the scene-hypergraph filtration induced by \textit{King Lear}. 
        This feature is convex in $\cat{Hgph}^=_m$, so the steady and ranging persistence diagrams are equal (see~\Cref{coro:convex-balanced-equivalence}).
        In the scene-hypergraph, the exclusivity feature represents the scenes that feature a unique character.
        This persistence diagram is rich because most of the scenes introduce a new character.}

        \centering
        \includegraphics[width = 0.7\textwidth]{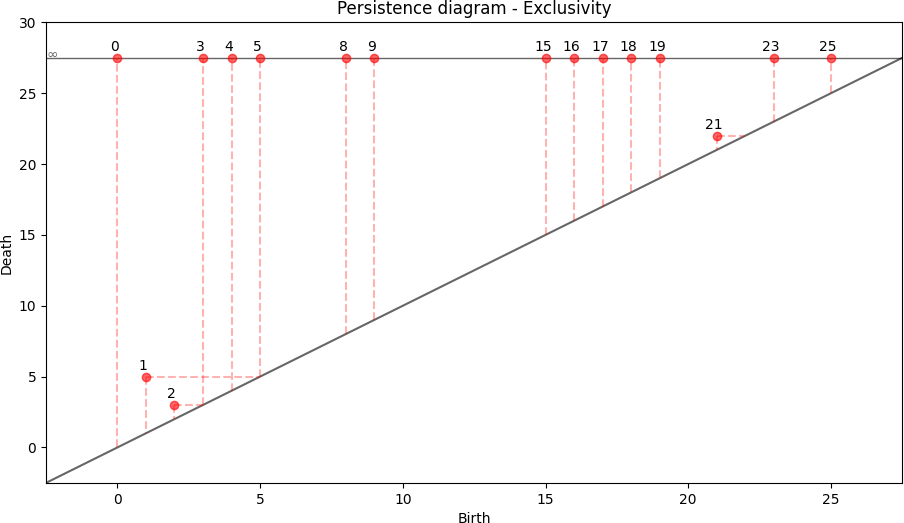}
        \label{fig:steady-exclusivity-scene-king_lear}
    \end{figure}

    \begin{figure}[H]
        \caption{Persistence of $\F^O$ for the scene-hypergraph filtration induced by \textit{King Lear}.
        This feature is convex in $\cat{Hgph}^=_m$ so the steady and ranging persistence diagrams are equal (see~\Cref{coro:convex-balanced-equivalence}).
        This diagram is sparse because the only max-original scenes are the 0-th, the 16-th and the 25-th scene.}

        \centering
        \includegraphics[width = 0.7\textwidth]{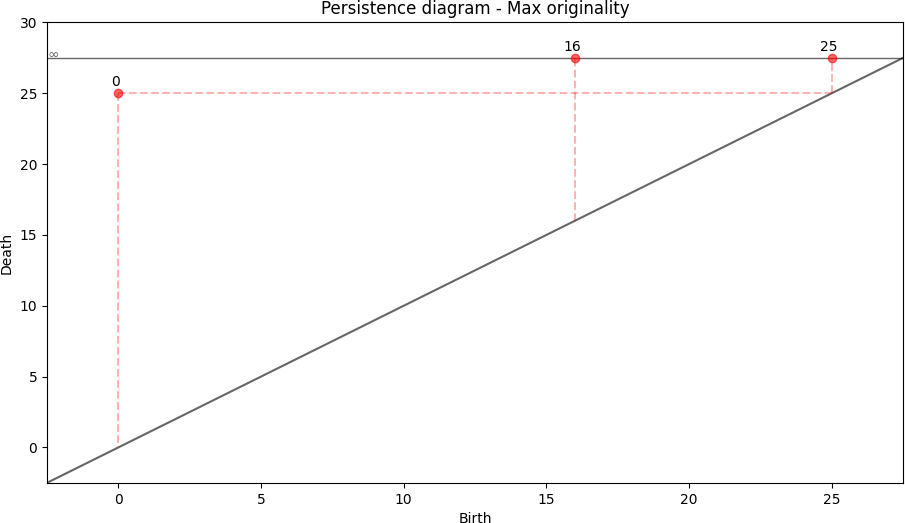}
        \label{fig:steady-max-originality-scene-king_lear}
    \end{figure}
    
    \smallbreak
    \paragraph{Results for the character-hypergraph of \textit{King Lear}}
    \begin{figure}[H]
        \centering
        \includegraphics[width = \textwidth]{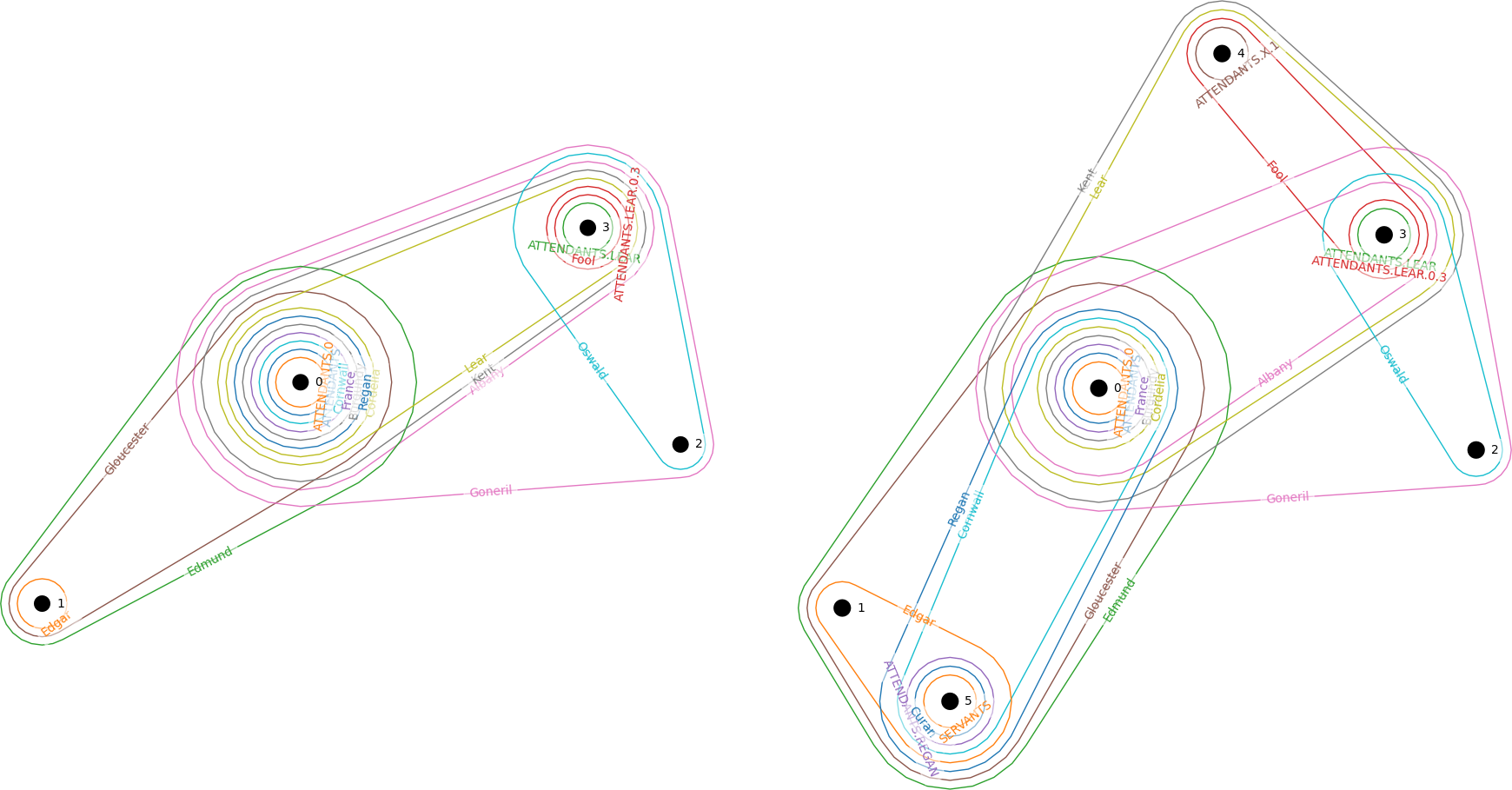}
          \caption{The character-hypergraph filtration of \textit{King Lear} at $t=3$ and $t=5$.
          The second hypergraph is the dual of the first hypergraph of~\Cref{fig:filtration-scene}.}
          \label{fig:filtration-charac}
    \end{figure}

    \begin{figure}[H]
        \caption{Steady persistence of the hub feature $\F^h$ for the character-hypergraph filtration induced by \textit{King Lear}.
        The character hubs are \textit{Kent}, \textit{Gloucester} and \textit{Goneril}.
        \textit{Kent} stops being a hub between $t=8$ and $t=9$.
        This induces a difference between the steady and ranging diagrams, which illustrates the non-convexity of $\F^h$ in $\cat{Hgph}^\leq_m$.}

        \centering
        \includegraphics[width =0.7\textwidth]{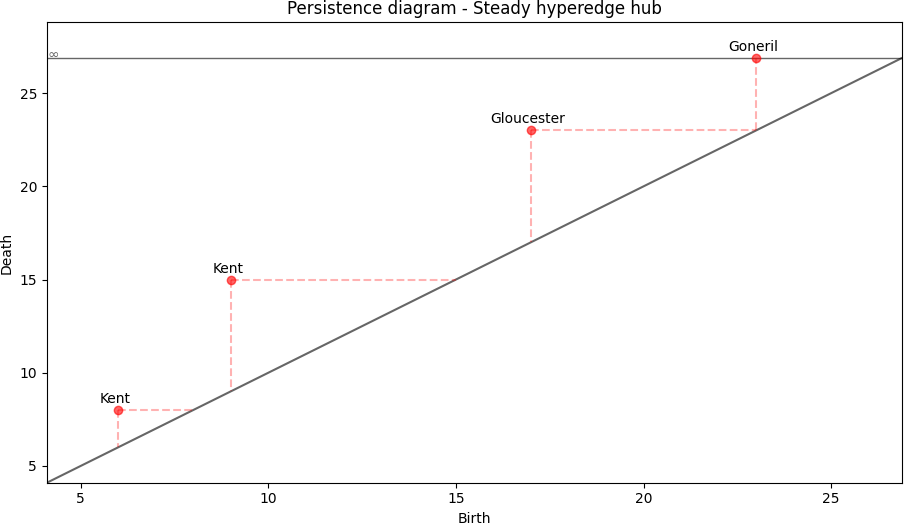}
        \label{fig:steady-hubs-character-king_lear}
    \end{figure}

    \begin{figure}[H]
        \caption{Ranging persistence of the hub feature $\F^h$ for the character-hypergraph filtration induced by \textit{King Lear}.}

        \centering
        \includegraphics[width =0.7\textwidth]{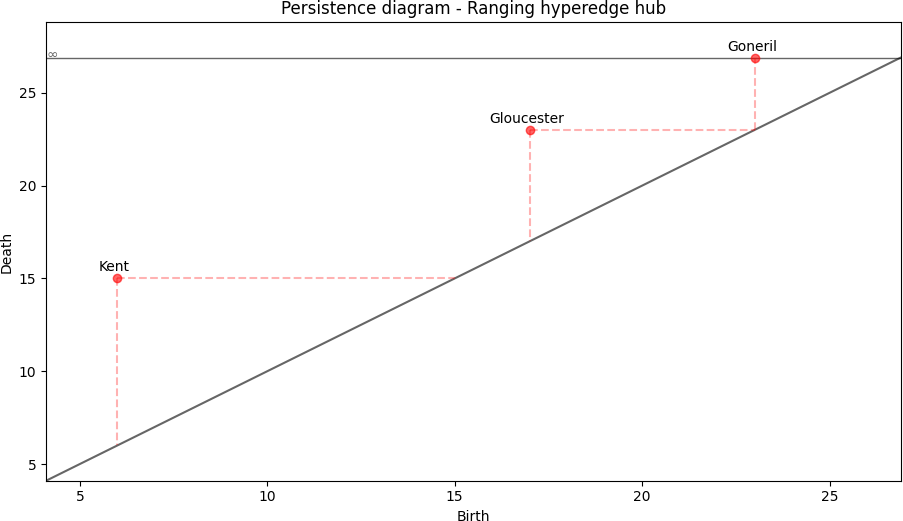}
        \label{fig:ranging-hubs-character-king_lear}
    \end{figure}

    \begin{figure}[H]
        \caption{Steady persistence of $\F^x$ for the character-hypergraph filtration induced by \textit{King Lear}.
        In the character-hypergraph, the exclusivity feature represents the characters that are the only character for some scene, i.e.\ the characters who have a monologue scene.
        In this play, only \textit{Edgar} has a monologue scene.
        Although $\F^x$ is not convex in $\cat{Hgph}^\leq_m$ (see~\Cref{fig:hyper-filtration-2}), the steady and ranging persistence diagrams for this filtration are equal.}

        \centering
        \includegraphics[width = 0.7\textwidth]{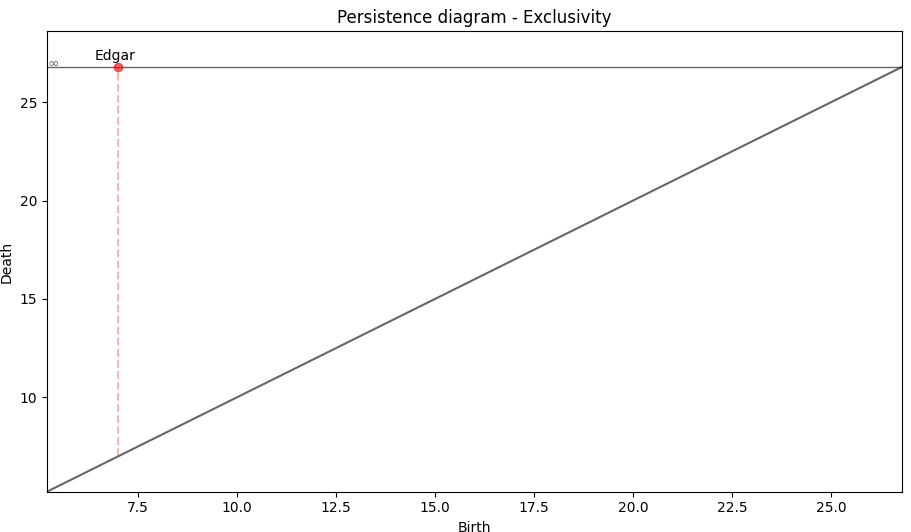}
        \label{fig:steady-exclusivity-character-king_lear}
    \end{figure}

    \begin{figure}
        \centering
        \includegraphics[width = 0.7\textwidth]{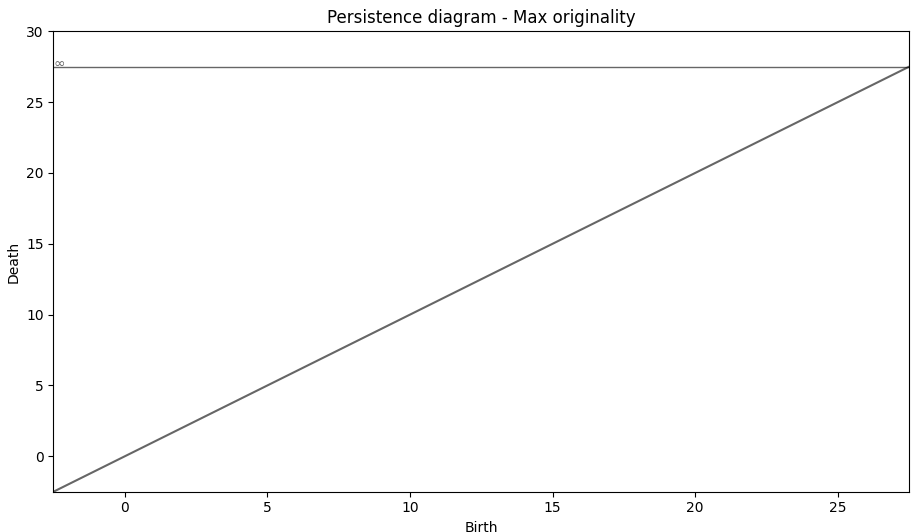}
        \caption{Steady persistence of $\F^O$ for the character-hypergraph filtration induced by \textit{King Lear}.
        The persistence diagram is empty, which means that no character is max-original, i.e.\ for every character $e$, there is another character $e'$ that appears in more than half of the scenes of $e$.
        Although $\F^O$ is not convex in $\cat{Hgph}^\leq_m$ (see~\Cref{fig:hyper-filtration-2}), the steady and ranging persistence diagrams for this filtration are equal.}
        \label{fig:steady-max-originality-character-king_lear}
    \end{figure}

\section{Conclusion}\label{sect:conclusion}

  %


    Steady and ranging persistence is a data analysis method based on generalized persistence and introduced in~\cite{bergomi-steady-ranging} to study features of graph filtrations.
    The present paper extends and investigates this theory.
    Our contribution is twofold.
    On the one hand, we extended the scope of steady and ranging persistence to consider different objects, such as simplicial complexes and hypergraphs.
    More precisely, we extended this persistence theory from graphs to finitely concrete categories and presented results and examples for hypergraph categories.
    We hope that our viewpoint will provide new insights into persistent data analysis.
    
    On the other hand, we investigated the stability of steady and ranging persistence.
    \add{First, we analyzed to what extent balancedness is equivalent to stability.
    We show that for tame filtrations, stability implies balancedness, and the converse holds under a condition on the category of interest (see~\Cref{prop:tame-stable-implies-tame-balanced,prop:tame-balanced-implies-tame-stable-condition}).
    Second, we presented a characterization of the features that induce balanced generators (see our main result~\Cref{coro:convex-balanced-equivalence}).
    This characterization is stated directly in terms of the features themselves, without requiring reference to interleaved filtrations or their persistence functions.
    By doing so, it resolves the open question posed in~\cite{bergomi-steady-ranging} about finding a general hypothesis on $\F$ so that $\sigma^\F$ or $\rho^\F$ is balanced.}
    
    In addition, we demonstrate that the balancedness of the steady or ranging persistence of a feature is equivalent to the equality of its steady and ranging generators.
    In general, ranging persistence can always be recovered from steady persistence, whereas the converse is not necessarily true.
    This means that steady persistence diagrams are more ``diverse'' and carry more information than ranging diagrams.
    An interesting consequence of our main result is that no feature can induce steady persistence diagrams that are both stable and strictly more informative than their corresponding ranging diagrams.
    
    To extend this work further, one may want to relax the requirements for the category of interest.
    For instance, only~\Cref{thm:ranging-balanced-implies-convex} requires the existence of an empty-like element, so this assumption may probably be discarded.
    Similarly, it might be possible to relax the concreteness of the category, for example, by considering a finitely concrete functor $\Phi$ from $\cat C_m$ to $\cat{FinSet}$ instead of a functor to $\cat{FinSet}_m$.
    \add{It would also be relevant to study if balanced ip-generators are stable when the triangle condition is not satisfied.}
    Finally, it may be interesting to study the universality (see~\cite{lesnick-theory-interleaved-distance-multidimensional}) of the steady and ranging generators.
    
    \paragraph{Acknowledgments.}
    First, I would like to express my gratitude to Massimo Ferri for his wise advice and great support, and to Patrizio Frosini for his friendly welcome and supervision at the University of Bologna.
    I am also grateful to Eliot Tron and Clemens Bannwart for the insightful math talks we had.
    Finally, I would like to thank the reviewers, whose comments greatly improved this paper.
    %

    \section*{Statements and Declarations}
    %
    %
    \paragraph{Competing interests.}
     I declare that the authors have no competing interests that might be perceived to influence the results and/or discussion reported in this paper. 

\bibliographystyle{plain}
\bibliography{main}

@Article{bergomi-beyond-topological-persistence,
AUTHOR = {Bergomi, Mattia G. and Ferri, Massimo and Vertechi, Pietro and Zuffi, Lorenzo},
TITLE = {Beyond Topological Persistence: Starting from Networks},
JOURNAL = {Mathematics},
VOLUME = {9},
YEAR = {2021},
NUMBER = {23},
ARTICLE-NUMBER = {3079},
URL = {https://www.mdpi.com/2227-7390/9/23/3079},
ISSN = {2227-7390},
DOI = {10.3390/math9233079}
}

@article{bergomi-steady-ranging,
  title={Steady and ranging sets in graph persistence},
  author={Bergomi, Mattia G. and Ferri, Massimo and Tavaglione, Antonella},
  journal={Journal of Applied and Computational Topology},
  year={2023},
  volume={7},
  pages={33 - 56}
}

@article{amico-natural-pseudo-distance,
  title={Natural Pseudo-Distance and Optimal Matching between Reduced Size Functions},
  author={D'Amico, Michele and Frosini, Patrizio and Landi, Claudia},
  journal={Acta Applicandae Mathematicae},
  year={2008},
  volume={109},
  pages={527-554},
  url={https://api.semanticscholar.org/CorpusID:1704971}
}

@article{bergomi-rank-based-persistence,
  title={Rank-based persistence},
  author={Bergomi, Mattia G. and Vertechi, Pietro},
  journal={Theory and Applications of Categories},
  year={2020},
  pages = {228-260},
  volume={35(9)}
}

@article{bergomi-exploring-graph-persistence,
AUTHOR = {Bergomi, Mattia G. and Ferri, Massimo},
TITLE = {Exploring Graph and Digraph Persistence},
JOURNAL = {Algorithms},
VOLUME = {16},
YEAR = {2023},
NUMBER = {10},
ARTICLE-NUMBER = {465},
URL = {https://www.mdpi.com/1999-4893/16/10/465},
ISSN = {1999-4893},
DOI = {10.3390/a16100465}
}

@article{carlsson-zigzag,
  title={Zigzag Persistence},
  author={Gunnar E. Carlsson and Vin de Silva},
  journal={Foundations of Computational Mathematics},
  year={2008},
  volume={10},
  pages={367-405},
  url={https://api.semanticscholar.org/CorpusID:6668405}
}

@article{carlsson-multidimensional-persistence,
author = {Carlsson, Gunnar and Zomorodian, Afra},
year = {2007},
month = {06},
pages = {71-93},
title = {The Theory of Multidimensional Persistence},
volume = {42},
journal = {Discrete and Computational Geometry},
doi = {10.1007/s00454-009-9176-0}
}

@article{biasotti-multidimensional-size-function,
author = {Biasotti, Silvia and Cerri, Andrea and Frosini, Patrizio and Giorgi, Daniela and Landi, Claudia},
year = {2008},
month = {10},
pages = {161-179},
title = {Multidimensional Size Functions for Shape Comparison},
volume = {32},
journal = {Journal of Mathematical Imaging and Vision},
doi = {10.1007/s10851-008-0096-z}
}

@inproceedings{chazal-proximity-persistence-module,
author = {Chazal, Fr\'{e}d\'{e}ric and Cohen-Steiner, David and Glisse, Marc and Guibas, Leonidas J. and Oudot, Steve Y.},
title = {Proximity of persistence modules and their diagrams},
year = {2009},
isbn = {9781605585017},
publisher = {Association for Computing Machinery},
address = {New York, NY, USA},
url = {https://doi.org/10.1145/1542362.1542407},
doi = {10.1145/1542362.1542407},
pages = {237–246},
numpages = {10},
location = {Aarhus, Denmark},
series = {SCG '09}
}

@article{bubenik-categorification-persistent-homology,
author = {Bubenik, Peter and Scott, Jonathan A.},
	doi = {10.1007/s00454-014-9573-x},
	url = {https://doi.org/10.1007%2Fs00454-014-9573-x},
	year = 2014,
	month = {jan},
	publisher = {Springer Science and Business Media {LLC}},
  	volume = {51},
  	number = {3},
  	pages = {600--627},
  	title = {Categorification of Persistent Homology},
	journal = {Discrete and Computational Geometry}
}

@article{bubenik-metrics-generalized-persistence,
  author={Bubenik, Peter and de Silva, Vin and Scott, Jonathan},
   title={Metrics for Generalized Persistence Modules},
   volume={15},
   ISSN={1615-3383},
   url={http://dx.doi.org/10.1007/s10208-014-9229-5},
   DOI={10.1007/s10208-014-9229-5},
   number={6},
   journal={Foundations of Computational Mathematics},
   publisher={Springer Science and Business Media LLC},
   year={2014},
   month=oct, pages={1501–1531}
}

@article{patel-generalized-diagrams,
   title={Generalized persistence diagrams},
   volume={1},
   ISSN={2367-1734},
   url={http://dx.doi.org/10.1007/s41468-018-0012-6},
   DOI={10.1007/s41468-018-0012-6},
   number={3–4},
   journal={Journal of Applied and Computational Topology},
   publisher={Springer Science and Business Media LLC},
   author={Patel, Amit},
   year={2018},
   month=may, pages={397–419}
  }

@inproceedings{cohen-steiner-stability-persistence,
author = {Cohen-Steiner, David and Edelsbrunner, Herbert and Harer, John},
year = {2005},
month = {06},
pages = {263-271},
title = {Stability of Persistence Diagrams},
volume = {37},
journal = {Discrete and Computational Geometry - DCG},
doi = {10.1007/s00454-006-1276-5}
}

@article{lesnick-theory-interleaved-distance-multidimensional,
   title={The Theory of the Interleaving Distance on Multidimensional Persistence Modules},
   volume={15},
   ISSN={1615-3383},
   url={http://dx.doi.org/10.1007/s10208-015-9255-y},
   DOI={10.1007/s10208-015-9255-y},
   number={3},
   journal={Foundations of Computational Mathematics},
   publisher={Springer Science and Business Media LLC},
   author={Lesnick, Michael},
   year={2015},
   month=mar, pages={613–650}
   }

@article{grigoryan-graphs-simplical-complexes-2014,
  title={Graphs associated with simplicial complexes},
  author={Alexander Grigor’yan and Yuri V. Muranov and Shing-Tung Yau},
  journal={Homology, Homotopy and Applications},
  year={2014},
  volume={16},
  pages={295-311}
}

@article{grigoryan-path-complexes-2020,
  title={Path Complexes and their Homologies},
  author={Alexander Grigor’yan and Yong Lin and Yuri V. Muranov and Shing-Tung Yau},
  journal={Journal of Mathematical Sciences},
  year={2020},
  volume={248},
  pages={564-599}
}

@article{grigoryan-homology-digraph-2021,
  title={Homology of Digraphs},
  author={Alexander Grigor’yan and Yuri V. Muranov and Rolando Jim{\'e}nez},
  journal={Mathematical Notes},
  year={2021},
  volume={109},
  pages={712 - 726}
}

@misc{bressan-homology-hypergraph,
  title={The embedded homology of hypergraphs and applications},
  author={St{\'e}phane Bressan and Jingyan Li and Shiquan Ren and Jie Wu},
  journal={Asian Journal of Mathematics},
  year={2016},
  url={https://api.semanticscholar.org/CorpusID:119141581}
}

@dataset{hyperbard,
  author       = {Corinna Coupette and
                  Jilles Vreeken and
                  Bastian Rieck},
  title        = {Hyperbard (Dataset)},
  month        = jun,
  year         = 2022,
  publisher    = {Zenodo},
  version      = {0.0.1},
  doi          = {10.5281/zenodo.6627159},
  url          = {https://doi.org/10.5281/zenodo.6627159}
}

@article{shibayama-originality,
author = {Shibayama, Sotaro and Wang, Jian},
year = {2019},
month = {10},
pages = {},
title = {Measuring Originality in Science},
volume = {122},
journal = {Scientometrics},
doi = {10.1007/s11192-019-03263-0}
}

@book{vanoosten-basic-category-theory,
  title={Basic Category Theory},
  author={Van~Oosten, Jaap},
  url={https://books.google.fr/books?id=6iDtcQAACAAJ},
  publisher={University of Aarhus. Basic Research in Computer Science [BRICS]},
  address={Aarhus, DK},
  year={1995}
}

@misc{leinster-basic-category-theory,
      title={Basic Category Theory},
      author={Tom Leinster},
      year={2016},
      eprint={1612.09375},
      archivePrefix={arXiv},
      primaryClass={math.CT},
      url={https://arxiv.org/abs/1612.09375},
}

@article{ouvrard-hypergraphs,
  author       = {Xavier Ouvrard},
  title        = {Hypergraphs: an introduction and review},
  journal      = {CoRR},
  volume       = {abs/2002.05014},
  year         = {2020},
  url          = {https://arxiv.org/abs/2002.05014},
  eprinttype    = {arXiv},
  eprint       = {2002.05014},
  bibsource    = {dblp computer science bibliography, https://dblp.org}
}

@article{dorfler-category-hypergraphs,
author={D{\"o}rfler, W. and Waller, D. A.},
title={A category-theoretical approach to hypergraphs},
journal={Archiv der Mathematik},
year={1980},
month={Dec},
day={01},
volume={34},
number={1},
pages={185-192},
issn={1420-8938},
doi={10.1007/BF01224952},
url={https://doi.org/10.1007/BF01224952}
}

@article{hypernetx,
author = {Praggastis, Brenda and Aksoy, Sinan and Arendt, Dustin and Bonicillo, Mark and Joslyn, Cliff and Purvine, Emilie and Shapiro, Madelyn and Yun, Ji Young},
doi = {10.21105/joss.06016},
journal = {Journal of Open Source Software},
month = mar,
number = {95},
pages = {6016},
title = {{HyperNetX: A Python package for modeling complex network data as hypergraphs}},
url = {https://joss.theoj.org/papers/10.21105/joss.06016},
volume = {9},
year = {2024}
}

@article{liu-hypergraph-homology,
author = {Liu, Xiang and Feng, Huitao and Wu, Jie and Xia, Kelin},
title = {Computing hypergraph homology},
journal = {Foundations of Data Science},
volume = {6},
number = {2},
pages = {172-194},
year = {2024},
issn = {},
doi = {10.3934/fods.2024007},
url = {https://www.aimsciences.org/article/id/65f00bfdc26d215a6050a126}
}

@article{liu-persistent-spectral-hypergraph-learning,
    author = {Liu, Xiang and Feng, Huitao and Wu, Jie and Xia, Kelin},
    title = {Persistent spectral hypergraph based machine learning (PSH-ML) for protein-ligand binding affinity prediction},
    journal = {Briefings in Bioinformatics},
    volume = {22},
    number = {5},
    pages = {bbab127},
    year = {2021},
    month = {04},
    issn = {1477-4054},
    doi = {10.1093/bib/bbab127},
    url = {https://doi.org/10.1093/bib/bbab127},
    eprint = {https://academic.oup.com/bib/article-pdf/22/5/bbab127/40260926/bbab127.pdf},
}

@article{babu-persistent-hypergraph,
author = {Babu, Archana and John, Sunil},
year = {2024},
month = {04},
pages = {},
title = {Persistent homology based Bottleneck distance in hypergraph products},
volume = {9},
journal = {Applied Network Science},
doi = {10.1007/s41109-024-00617-3}
}

@misc{ys-github-steady-ranging-persistence,
  author = {Yann-Situ},
  title = {Steady and Ranging Persistence for Hypergraphs},
  year = {2024},
  publisher = {GitHub},
  journal = {GitHub repository},
  url = {https://github.com/Yann-Situ/Hypergraph-Steady-Ranging-Persistence},
  howpublished = {\url{https://github.com/Yann-Situ/Hypergraph-Steady-Ranging-Persistence}}
}

\begin{appendices}

\section{Proofs}\label{sect:appendix-proof}
\subsection{Proofs of~\Cref{sect:prior-works}}
\begin{delayedproof}{lem:epsilon-close-charac}
\ \\
    \noindent {$\implies$.}\quad
    There exist $(\phi_w:F_w\to G_{w+\epsilon})_{w\in\R}$ and $(\psi_w:G_w\to F_{w+\epsilon})_{w\in\R}$ two families of monomorphisms such that for any $w\in\R$, the diagrams of~\Cref{def:epsilon-close} commute.

    Denote $N := \max (\sup f, \sup g)$.
    $G_u^v$ and $F_u^v$ are graph inclusions in $X$ and $X'$, so for every $u$ and $v$ greater than $N$ we have $\phi_u(x)=\phi_v(x)$ for every node or edge $x$ in $X$.
    Similarly, $\psi_u(x')=\psi_v(x')$ for all $x'$ in $X'$.
    
    Hence, $(\phi_u)_{u\in\R}$ and $(\psi_u)_{u\in\R}$ induce two morphisms $\phi:X\to X'$ and $\psi:X'\to X$ defined as $\phi = \phi_{N}$ and $\psi = \psi_{N}$.
    By commutativity, we have that $\phi\circ \psi = \id_{X'}$ and $\psi\circ \phi = \id_X$, so they are isomorphisms.
    
    Consider $x\in X$.
    We have $x\in F_{f(x)}$ and $\phi(x)\in G_{g(\phi(x))}$.
    By applying $g\circ \phi$ to $x$ and $f\circ \psi$ to $\phi(x)$, we obtain
    $$g(\phi(x)) \leq f(x)+\epsilon \qquad f(\psi(\phi(x))) \leq g(\phi(x))+\epsilon$$
    
    As $\psi\circ \phi = \id_X$, the two last inequalities give
    $$ f(x)-\epsilon \leq g(\phi(x)) \leq f(x)+\epsilon$$
    As a result, $sup_{x\in X}|f(x)-g(\phi(x))|\leq \epsilon$ with $\phi:X\to X'$ an isomorphism.
    
    \smallskip

    \noindent {$\Longleftarrow$.}\quad
    We define $\phi_w:= \phi_{|F_w}$ and $\psi_w:= \phi\mun_{|G_w}$.
    
    First, we show that we have $\phi_w:F_w\to G_{w+\epsilon}$ and $\psi_w:G_w\to F_{w+\epsilon}$.
    Let $x\in F_w$.
    As $sup_{x\in X}|f(x)-g(\phi(x))|\leq \epsilon$, we have that $g(\phi_w(x))\leq f(x)+\epsilon \leq w+\epsilon$.
    Therefore, $\phi_w(x)\in G_{w+\epsilon}$.
    Similarly, $x\in G_w$ implies $\psi_w(x)\in F_{w+\epsilon}$.
    
    Second, we prove that the diagrams commutes.
    Let $x$ be an element of $F_{w-\epsilon}$.
    We have $\psi_{w}\circ \phi_{w-\epsilon}(x) = \phi\mun_{|G_w}(\phi_{|F_{w-\epsilon}}(x)) = x = F_{w-\epsilon}^{w+\epsilon}(x)$.
    Let $x$ be an element of $F_{u}$.
    We have $\phi_v \circ F_u^v(x)= \phi(x) = G_{u+\epsilon}^{v+\epsilon}\circ \phi_u(x)$.
    
    By symmetry of $F$ and $G$, the four diagrams commute.

\end{delayedproof}

\subsection{Proofs of~\Cref{sect:extending-steady-ranging}}
\add{
\begin{delayedproof}{lem:tame-implies-finite-persistence}
Let $a_1 < a_2 < \dots < a_n$ be the sequence of critical values of $F$, and set $a_0=-\infty$ and $a_{n+1} = +\infty$.
\addd{We set $\pi$ to be either $\sigma^\F$ or $\rho^\F$.}

Lemma 4.4 of~\cite{bubenik-categorification-persistent-homology} implies that $F_w^{w'}$ is an isomorphism if $w$ and $w'$ are between the same consecutive critical values.
Hence, there is a natural isomorphism between $F$ and a filtration $\bar F$ which is constant on the intervals $]a_i,a_{i+1}[$: $\bar F_w^{w'} = \id$ if $w$ and $w'$ are in $]a_i,a_{i+1}[$.
Therefore, for fixed $\bar u$ and $\bar v$ in $\R$, the following functions are constant on every interval $]a_i,a_{i+1}[$ (see~\Cref{def:steady-ranging-set}):
$$u \leq \bar v \mapsto S_{\bar F}^\F(u\leq \bar v) \quad v \geq \bar u \mapsto S_{\bar F}^\F(\bar u\leq v) \quad u \leq \bar v \mapsto R_{\bar F}^\F(u\leq \bar v) \quad v \geq \bar u \mapsto R_{\bar F}^\F(\bar u\leq v)$$

As a result, $\pi(\bar F)$ is constant on all the open rectangles $]a_i,a_{i+1}[\times]a_j,a_{j+1}[$.
It implies that all cornerpoints of $Dgm(\pi(\bar F))$ outside the diagonal are in the finite grid $\{a_0,\dots, a_{n+1}\}^2$.
Indeed, being a point $(u,v)$ (with $u<v$) inside an open rectangle $]a_i,a_{i+1}[\times]a_j,a_{j+1}[$ or inside an open segment $\{a_i\}\times]a_j,a_{j+1}[$ (or $]a_i,a_{i+1}[\times \{a_j\}$) would imply $\mu(u,v) = 0$ by~\Cref{def:persistence-multiplicity}.

\addd{It remains to prove that cornerpoints above the diagonal have finite multiplicity.
This is due to the fact that persistence functions only take finite values, which is an implicit requirement in the definition of a persistence function (see~\Cref{rem:finite-multiplicities}).
Recall that, in our case, the persistence functions obtained from steady and ranging persistence take finite values because of the finiteness of $S_{\bar F}^\F(u \leq v)$ and $R_{\bar F}^\F(u \leq v)$ (due to the finitely concrete functor).}

All in all, the persistence diagram of $\pi(\bar F)$ is finite.
As ip-generators are invariant under natural isomorphisms (see~\Cref{def:ip-generator}), $\pi(F) = \pi(\bar F)$ is finite.
\end{delayedproof}
}

\add{
\begin{delayedproof}{prop:tame-stable-implies-tame-balanced}
Let $F$ and $G$ be two $\epsilon$-interleaved tame filtrations.
We prove that $\pi(F)$ and $\pi(G)$ are $\epsilon$-compatible.

As $\pi$ is tame-stable, we have that the bottleneck distance is inferior to the interleaving distance, hence, $d_B(\pi(F), \pi(G)) \leq \epsilon$.
Moreover, $F$ is tame so $Dgm(\pi(F))$ has a finite number of cornerpoints outside the diagonal (see~\Cref{lem:tame-implies-finite-persistence}), so there exists a bijection $\gamma$ between the persistence diagrams $Dgm_F := Dgm(\pi(F))$ and $Dgm_G := Dgm(\pi(G))$ realizing the bottleneck distance: $\sup_{p \in Dgm_F}||p - \gamma(p)||_\infty = d_B(\pi(F), \pi(G))$, in particular, $\forall p\in Dgm_F, ||p - \gamma(p)||_\infty \leq \epsilon$.
This implies the following inclusion:
\begin{equation}\label{eq:tame-stable-implies-tame-balanced-1}
  \left\{ \gamma(p)\ /\ p\in Dgm_F\ \tx{with}\begin{cases}
       p_u < u-\epsilon\\
       p_v > v+\epsilon
      \end{cases}
      \right\} \subseteq 
    \left\{ p'\in Dgm_G\ \tx{with}\begin{cases}
       p'_u < u\\
       p'_v > v
      \end{cases}
    \right\}
\end{equation}
where $(p_u, p_v)$ denotes the coordinates of a point $p$ in a persistence diagram (considered as a multiset).
Then, fix $\bar u \leq \bar v \in \R$ and denote $\mu_F$ and $\mu_G$ the multiplicity functions of the diagrams $Dgm_F$ and $Dgm_G$ respectively.
\Cref{lem:radar,eq:tame-stable-implies-tame-balanced-1} gives:
\begin{align*}
    \pi(F)(\bar u - \epsilon \leq \bar v + \epsilon) &= \sum_{u<\bar u-\epsilon,\ v>\bar v+\epsilon} \mu_F(u,v)\\
    &= \left|\left\{ p\in Dgm_F\ \tx{with}\begin{cases}
         p_u < \bar u-\epsilon\\
         p_v > \bar v+\epsilon
        \end{cases}
        \right\}\right|\\
    &\leq \left|\left\{ p'\in Dgm_G\ \tx{with}\begin{cases}
       p'_u < \bar u\\
       p'_v > \bar v
      \end{cases}
    \right\}\right|\\
    &\leq \sum_{u<\bar u,\ v>\bar v} \mu_G(u,v) = \pi(G)(\bar u\leq\bar v)
\end{align*}
As a result, $\pi(F)$ and $\pi(G)$ are $\epsilon$-compatible.
\end{delayedproof}
}

\add{
\begin{delayedproof}{lem:interleaved-equiv}\ \\
\noindent\textbf{$\implies$.}\quad
Let $(\phi_w)_{w\in\R}$ and $(\psi_w)_{w\in\R}$ be the function of the $\epsilon$-interleaving between $F$ and $G$: $\phi_w : F_w \injto G_{w+\epsilon}$ and $\psi_w : G_w \injto F_{w+\epsilon}$.

Using $(\phi_w)_{w\in\R}$, we get that $G_\infty$ (along with the morphisms $(G_{w+\epsilon}^\infty \circ \phi_w)_{w\in\R}$) is a colimit of $F$.
By definition of $F_\infty$, we have $F_\infty = F_a$ for some $a\in\R$, so we define $\phi := G_{a+\epsilon}^\infty \circ \phi_a$ a monomorphism from $F_\infty$ to $G_\infty$.

We first show that $\phi$ is an isomorphism.
Colimits are unique up to isomorphism (see~\cite{leinster-basic-category-theory}), hence, we have that 
$F_\infty$ and $G_\infty$ are isomorphic.
Therefore, $\Phi\phi$ is a bijection because it is injective from $\Phi F_\infty$ to $\Phi G_\infty$, which have the same cardinal.
As a result, there exists $\chi := (\Phi\phi)^{-1} : \Phi G_\infty \to \Phi F_\infty$.
Using the triangle condition with $\Phi(\id_{G_\infty}) = \Phi\phi \circ \chi$ we obtain a monomorphism $\psi:G_\infty\to F_\infty$ such that $\id_{G_\infty} = \phi\circ\psi$.
Moreover, $\phi\circ \id_{F_\infty} = \id_{G_\infty}\phi  = \phi\circ\psi\phi$, which implies $\id_{F_\infty} = \psi\circ\phi$ by monomorphism property of $\phi$. 
Hence, $\phi$ is an isomorphism between $F_\infty$ and $G_\infty$.

Moreover, as $\phi_a F_w^a = G_{w+\epsilon}^a \phi_w$ (see~\Cref{def:epsilon-close}), we have:
\begin{equation}\label{eq:interleaved-equiv-1}
  \forall w,\quad G_{w+\epsilon}^\infty \circ \phi_w = \phi \circ F_w^\infty
\end{equation}

We now show that $\forall x \in \Phi F_\infty,\ |f_F(x) - f_G \circ \Phi\phi (x)| \leq \epsilon$.\\      
Let $x \in \Phi F_\infty$ such that $x\in \im\Phi F_u^\infty$.
There exists $y \in \Phi F_u$ such that $x = \Phi F_u^\infty (y)$.
Then,~\Cref{eq:interleaved-equiv-1} gives $\Phi\phi(x) = \Phi G_{u+\epsilon}^\infty ( \phi_u(y))$, therefore, $\phi(x) \in \im\Phi G_{u+\epsilon}^\infty$.

This implies that $f_G(\phi(x)) \leq f_F(x)+\epsilon$.
Using a similar reasoning on $\phi(x)$, we get $f_F(x) = f_F(\phi^{-1}(\phi(x))) \leq f_G(\phi(x))+\epsilon$.
As a result, $|f_F(x) - f_G \circ \Phi\phi (x)| \leq \epsilon$.

\medskip

\noindent\textbf{$\Longleftarrow$.}\quad
Let $x \in \Phi F_w$.
We define $x_\infty := \Phi F_w^\infty(x) \in \im\Phi F_w^\infty$, so $f_F(x_\infty) \leq w$.
As $||f_F - f_G \circ \Phi\phi||_\infty \leq \epsilon$, we have that $f_G(\Phi\phi(x_\infty)) \leq w+\epsilon$.
Hence, $\Phi\phi(x_\infty) \in \im\Phi G_{w+\epsilon}^\infty$.
Therefore, there exists a $y_x \in \Phi G_{w+\epsilon}$ such that $\Phi\phi(x_\infty) = \Phi G_{w+\epsilon}^\infty (y_x)$.
Such a $y_x$ is unique because $\Phi G_{w+\epsilon}^\infty$ is injective.

As a result, denote $\chi_w : x\in \Phi F_w \mapsto y_x \in \Phi G_{w+\epsilon}$.
$\chi$ is injective because $y_x = y_{x'} \implies \Phi\phi(x_\infty) = \Phi\phi(x'_\infty) \implies x_\infty = x'_\infty \implies x = x'$ as $\Phi\phi$, $\Phi F_w^\infty$ and $\Phi G_{w+\epsilon}^\infty$ are injective.

Using the triangle condition with the following diagram
\begin{center}
  \begin{tikzcd}
\Phi F_w \arrow[rd, "\Phi(\phi \circ F_w^\infty)"'] \arrow[rr, "\chi_w"] & & \Phi G_{w+\epsilon} \arrow[ld, "\Phi G_{w+\epsilon}^\infty"] \\
& \Phi G_\infty & 
\end{tikzcd}
\end{center}
we obtain a monomorphism $\phi_w : F_w \injto G_{w+\epsilon}$ such that:
\begin{equation}\label{eq:interleaved-equiv-2}
  G_{w+\epsilon}^\infty \circ \phi_w =  \phi \circ F_w^\infty
\end{equation}
By symmetry, we can also build a monomorphism $\psi_w : G_w \injto F_{w+\epsilon}$ such that:
\begin{equation}\label{eq:interleaved-equiv-3}
  F_{w+\epsilon}^\infty \circ \psi_w =  \phi^{-1} \circ G_w^\infty
\end{equation}

\smallskip

We now prove that $(\phi_w)_{w\in\R}$ and $(\psi_w)_{w\in\R}$ define an $\epsilon$-interleaving between $F$ and $G$.
We only prove that the two first diagrams of~\Cref{def:epsilon-close} commutes, as the two others are similar.
\begin{align*}
    F_{w+e}^\infty \circ \psi_w \phi_{w-e}\ &= \phi^{-1} G_w^\infty \phi_{w-e}    & \tx{by~\Cref{eq:interleaved-equiv-3}}\\ 
    &= \phi^{-1}  \phi F_{w-e}^\infty  & \tx{by~\Cref{eq:interleaved-equiv-2}}\\
    &= F_{w+e}^\infty \circ F_{w-e}^{w+e}\\
    F_{w+e}^\infty \circ F_{u+e}^{w+e} \psi_u\ &= F_{u+e}^\infty \psi_u = \phi^{-1} G_u^\infty    & \tx{by~\Cref{eq:interleaved-equiv-3}}\\ 
    &= \phi^{-1} G_w^\infty G_u^w \\
    &= F_{w+e}^\infty \circ \psi_w G_u^w  & \tx{by~\Cref{eq:interleaved-equiv-3}} 
\end{align*}
As $F_{w+e}^\infty$ is a monomorphism, we obtain that $\psi_w \phi_{w-e} = F_{w-e}^{w+e}$ and $F_{u+e}^{w+e} \psi_u = \psi_w G_u^w$.
\end{delayedproof}
}

\subsection{Proofs of~\Cref{sect:convex-balanced}}

\begin{delayedproof}{prop:convex-feature-charac}
\ \\
    \noindent {$\implies$.}\quad
    By~\Cref{prop:steady-ranging-basic}(\emph{0.}), we only need to prove that $R^\F \subseteq S^\F$.
    
    Let $F$ be a filtration and $(A,F_u)\in R^\F_F(u\leq v)$.
    Let $x\leq u$, $y \geq v$ and $A'$ be such that $A=\Phi(F_x^u)(A')$, $(A',F_x) \in\F$ and $F_u^{y}(A,F_{u}) \in\F$.
    
    Let $w \in [u,v]$.
    We show that $F_u^w(A,F_u)\in\F$:
    
    we have $(A',F_x)\in\F$ and $F_u^{y}(A,F_u) = F_u^{y}F_x^{u}(A',F_x) = F_x^{y}(A',F_x)\in\F$.
    $F_w^y F_x^w = F_x^y$ so by convexity we have $F_x^w(A',F_x) = F_u^w(A,F_u)\in\F$.

\smallskip

    \noindent {$\Longleftarrow$.}\quad
    Let $A$, $\iota:X\injto X'$ and $\iota':X'\injto X''$ such that $(A,X)\in\F\ \tx{and}\ \iota'\iota(A,X)\in\F$.
    We build the following filtration $F$:
    \begin{align*}
        F_u =
        \begin{cases}
            X &\tx{if $u\leq 0$}\\
            X' &\tx{if $u\in]0,2[$}\\
            X'' &\tx{if $u\geq2$}
        \end{cases}
    \end{align*}
    We have $(A,X)\in R^\F_F(0\leq 2)$ (by definition of ranging set with $x=u=0$, $y=v=2$, $A'=A$ and because $F_0^2=\iota'\iota$).
    Therefore, $(A,X)\in S^\F_F(0\leq 2)$.
    By definition of steady sets, this implies that $F_0^{1}(A,X) = \iota(A,X)\in\F$.

\end{delayedproof}

\begin{delayedproof}{prop:simple-convex-equivalence}
We prove the proposition using logical equivalences.
Let $X_1\injto{\iota}X_2\injto{\iota'}X_3$ be a part of a filtration and let $A$ be a subset of $\Phi(X_1)$ (where $\Phi$ is the functor associated to the finitely concrete mono category, see~\Cref{def:finitely-concrete-category}).
\begin{align*}
&\biggl( \iota\iota'(A,X_1)\in\F \ \tx{and} \ \iota(A,X_1)\notin\F \implies (A,X_1)\notin \F \biggr)  \\
&\equi
    \biggl(  (A,X_1)\in \F  \implies \Bigl( \iota\iota'(A,X_1)\notin\F \Bigr) \lor \Bigl( \iota(A,X_1)\in\F  \Bigr) \biggr)\\
&\equi \Bigl( (A,X_1)\notin \F  \Bigr) \lor \Bigl( \iota\iota'(A,X_1)\notin\F \Bigr) \lor \Bigl( \iota(A,X_1)\in\F  \Bigr)\\
&\equi \lnot \biggl( \Bigl( (A,X_1)\in \F  \Bigr) \land \Bigl( \iota\iota'(A,X_1)\in\F \Bigr)  \biggr)  \lor \Bigl( \iota(A,X_1)\in\F  \Bigr)\\
&\equi \biggl( \Bigl( (A,X_1)\in \F  \Bigr) \land \Bigl( \iota\iota'(A,X_1)\in\F \Bigr) \implies \iota(A,X_1)\in\F   \biggr)
\end{align*}
As a consequence:
\begin{align*}
&\F \ \tx{is simple} \\
&\equi \forall (X_1\injto{\iota}X_2\injto{\iota'}X_3),\ \biggl( \iota\iota'(A,X_1)\in\F \ \tx{and} \ \iota(A,X_1)\notin\F \implies (A,X_1)\notin \F \biggr) \\
&\equi \forall (X_1\injto{\iota}X_2\injto{\iota'}X_3),\ \biggl( \Bigl( (A,X_1)\in \F  \Bigr) \land \Bigl( \iota\iota'(A,X_1)\in\F \Bigr) \implies \Bigl( \iota(A,X_1)\in\F  \Bigr)  \biggr)\\
&\equi \F \ \tx{is convex}
\end{align*}

\end{delayedproof}

\begin{delayedproof}{prop:max-min-ric-feature}
    Let $A$ and $X\injto{\iota}X'\injto{\iota'}X''$ such that $(A,X)\in M\F$ and $\iota'\iota(A,X)\in M\F$.
    We show that $\iota(A,X)\in M\F$.
    
    First, by convexity of $\F$ (because $\F$ is right-continued and~\Cref{prop:convex-feature-examples}) we have $\iota(A,X)\in \F$.
    Hence, it remains to show the maximality of $\Phi\iota(A)$ in $\Phi X'$ with respect to $\F$.

    Let $B\subseteq \Phi X'$ such that $(B,X')\in\F$ and $\Phi \iota(A)\subseteq B$.
    By applying $\Phi \iota'$, we have $\Phi(\iota'\iota)(A)\subseteq \Phi \iota'(B)$.
    As $\F$ is right-continued we obtain $\iota'(B,X')=(\Phi \iota'(B),X'') \in\F$.

    By maximality of $\Phi(\iota'\iota)(A)$ in $X''$  w.r.t. $\F$ (because $\iota'\iota(A,X)\in M\F$) we obtain $\Phi \iota'(B)=\Phi(\iota'\iota)(A)$.
    By injectivity of $\Phi\iota'$ we get $\Phi\iota(A)=B$.
    This concludes the maximality of $\Phi\iota(A)$ in $\Phi X'$ w.r.t. $\F$.
    The proof of the convexity of $m\F$ uses the same arguments.

\end{delayedproof}

\begin{delayedproof}{thm:convex-implies-balanced}
    Let $\F$ be a convex feature. 
    Let $F$ and $G$ be two $\epsilon$-interleaved filtrations (with their associated monomorphism families $\phi$ and $\psi$).
    We show that $\sigma^\F_F$ and $\sigma^\F_G$ are $\epsilon$-compatible.
    By symmetry, we only need to show one side of $\epsilon$-compatibility.
    Specifically, we will prove that there is an injection from $S^\F_F(u-\epsilon\leq v+\epsilon)$ to $S^\F_G(u\leq v)$.

    Let $(A,F_{u-\epsilon})\in S^\F_F(u-\epsilon\leq v+\epsilon)$ and $(B, G_u) := \phi_{u-\epsilon}(A,F_{u-\epsilon})$.
    Let $w\in[u,v]$.
    By commutativity of the diagrams in~\Cref{def:epsilon-close}, we have $G_u^w(B,G_u) = \phi_{w-\epsilon} F_{u-\epsilon}^{w-\epsilon}(A,F_{u-\epsilon})$.
    \begin{itemize}
        \item We have $F_{u-\epsilon}^{w-\epsilon}(A,F_{u-\epsilon})\in\F$ because $(A,F_{u-\epsilon})\in S^\F_F(u-\epsilon\leq v+\epsilon)$ and $w-\epsilon$ is in $[u-\epsilon,v+\epsilon]$.
        \item We also have $\psi_w\phi_{w-\epsilon} F_{u-\epsilon}^{w-\epsilon}(A,F_{u-\epsilon})\in\F$ because it is equal to $F_{u-\epsilon}^{w+\epsilon}(A,F_{u-\epsilon})$ by commutativity and because $w+\epsilon$ is in $[u-\epsilon,v+\epsilon]$.
    \end{itemize}
    By convexity of $\F$, we obtain $\phi_{w-\epsilon} F_{u-\epsilon}^{w-\epsilon}(A,F_{u-\epsilon})= G_u^w(B,G_u) \in\F$.
    As a consequence, $(B,G_u)\in S^\F_G(u\leq v)$ and $\phi_{u-\epsilon}$ induces the following map:
    \begin{align*}
    \bar\phi_{u-\epsilon} :\ S^\F_F(u-\epsilon\leq v+\epsilon) &\to\ \ S^\F_G(u\leq v)\\
    (A,F_{u-\epsilon})\quad &\mapsto\ (B,G_u)\ := \phi_{u-\epsilon}(A,F_{u-\epsilon})\ = (\Phi\phi_{u-\epsilon}(A),G_u)
    \end{align*}
    The injectivity of $\Phi\phi_{u-\epsilon}$ concludes the proof. 

\end{delayedproof}

\begin{delayedproof}{thm:steady-balanced-implies-convex}
    Let $\F$ be a non-convex feature.
    By non-convexity there exist $X\injto{\iota}X'\injto{\iota'}X''$ and $A\subseteq X$ such that $(A,X)\in\F$ and $\iota'\iota(A,X)\in\F$ but $\iota(A,X)\notin\F$.

    We define the two following filtrations $F$ and $G$:
    \begin{align*}
        F_u = \begin{cases}
            X &\tx{if $u< 3$}\\
            X' &\tx{if $u\in[3,5[$}\\
            X'' &\tx{if $u\geq5$}
        \end{cases}
        \qquad
        G_u = \begin{cases}
            X &\tx{if $u< 4$}\\
            X'' &\tx{if $u\geq4$}
        \end{cases}
    \end{align*}
    \add{where $F_u^v$ equals either $\id$, $\iota$, $\iota'$ or $\iota'\iota$ and $G_u^v$ equals either $\id$ or $\iota'\iota$ depending on $(u\leq v)$.}
    We also define the two families of monomorphism $(\phi_w:F_w\to G_{w+1})_{w\in\R}$ and $(\psi_w:G_w\to F_{w+1})_{w\in\R}$ that link $F$ and $G$:
    \begin{align*}
        \phi_w =  \begin{cases}
            \id &\tx{if $w< 3$ or $w\geq 5$}\\
            \iota' &\tx{if $u\in[3,5[$}
        \end{cases}
        \qquad
        \psi_w =  \begin{cases}
            \id &\tx{if $w< 2$ or $w\geq 4$}\\
            \iota &\tx{if $u\in[2,4[$}
        \end{cases}
    \end{align*}

    This construction implies that $F$ and $G$ are $1$-interleaved.
    We will not detail the proof of this assumption, but the idea is to show that every possible diagram of the definition of $\epsilon$-interleaved~\Cref{def:epsilon-close} commutes.
    Instead, we provide the summary diagram below:
    \begin{center}
\adjustbox{scale=0.83}{%
\begin{tikzcd}
F: & \dots \arrow[r, "id"]                                     & F_1=X \arrow[rd, "id" description] \arrow[rr, "\iota"] &                                                                             & F_3=X' \arrow[rd, "\iota'" description] \arrow[rr, "\iota'"] &                                                                  & F_5=X'' \arrow[r, "id"] \arrow[rd, "id"] & \dots \\
G: & \dots \arrow[rr, "id"] \arrow[ru, "id" description] &                                                                    & G_2=X \arrow[ru, "\iota" description] \arrow[rr, "\iota'\iota"] &                                                                                & G_4=X'' \arrow[ru, "id"] \arrow[rr, "id"] &                                                                 & \dots
\end{tikzcd}
}

    \end{center}

    Using the definition of steady-sets, we obtain the following characterizations:
    \begin{align*}
    S^\F_G(0\leq6) &= \left\{ (B,G_0)\ /\ \forall v \in [0,6],~G_0^v(B,G_0)\in\F \right\}\\
    &= \left\{ (B,X)\ /\  \begin{cases}
        \id(B,X) \in\F\\
        \iota'\iota(B,X) \in\F
    \end{cases} \right\}
    \\
    S^\F_F(1\leq5) &=  \left\{ (B,F_1)\ /\ \forall v \in [1,5],~F_1^v(B,F_1)\in\F \right\}\\
    &= \left\{ (B,X)\ /\  \begin{cases}
        \id(B,X) \in\F\\
        \iota(B,X) \in\F\\
        \iota'\iota(B,X) \in\F
    \end{cases} \right\}
    \end{align*}
    These characterizations directly imply that $S^\F_F(1\leq5)\subseteq S^\F_G(0\leq6)$.

    However, we have $(A,X)\in\F$ and $\iota'\iota(A,X)\in\F$ but $\iota(A,X)\notin\F$.
    This induces $(A,X)\in S^\F_G(0\leq6)\bs S^\F_F(1\leq5)$, therefore $S^\F_F(1\leq5) \subsetneq S^\F_G(0\leq6)$.
    As these two sets are finite (because they are in bijection with subsets of $\Phi(F_1)$ and $\Phi(G_0)$, which are finite), we obtain:
    $$ \left| S^\F_F(1\leq5) \right| < \left| S^\F_G(0\leq6) \right| $$

    As a result, $F$ and $G$ are $1$-interleaved but $\sigma^\F_F$ and $\sigma^\F_G$ are not $1$-compatible, hence $\F$ is not steady-balanced.

\end{delayedproof}

\begin{delayedproof}{thm:ranging-balanced-implies-convex}
    For any object $X$ of we denote $\vec{\emptyset}^X$ the one and only morphism from $\hat\emptyset$ to $X$ (as $\hat\emptyset$ is an initial object).

    Let $\F$ be a non-convex feature.
    By non-convexity there exist $X\injto{\iota}X'\injto{\iota'}X''$ and $A\subseteq X$ such that $(A,X)\in\F$ and $\iota'\iota(A,X)\in\F$ but $\iota(A,X)\notin\F$.

    We define the two following filtrations $F$ and $G$:
    \begin{align*}
        F_u = \begin{cases}
            \hat\emptyset &\tx{if $u< 1$}\\
            X &\tx{if $u\in[1, 3[$}\\
            X' &\tx{if $u\in[3,5[$}\\
            X'' &\tx{if $u\geq5$}
        \end{cases}
        \qquad
        G_u = \begin{cases}
            \hat\emptyset &\tx{if $u< 2$}\\
            X' &\tx{if $u \in[2,6[$}\\
            X'' &\tx{if $u\geq6$}
        \end{cases}
    \end{align*}
    \add{where $F_u^v$ equals either $\vec{\emptyset}^X$, $\id$, $\iota$, $\iota'$ or $\iota'\iota$ and $G_u^v$ equals either $\vec{\emptyset}^{X'}$, $\id$ or $\iota'$ depending on $(u\leq v)$.}

    We also define the two families of monomorphisms $(\phi_w:F_w\to G_{w+1})_{w\in\R}$ and $(\psi_w:G_w\to F_{w+1})_{w\in\R}$ that link $F$ and $G$:
    \begin{align*}
        \phi_w =  \begin{cases}
            \id &\tx{if $w< 1$ or $w\geq 3$}\\
            \iota &\tx{if $w\in[1,3[$}
        \end{cases}
        \qquad
        \psi_w =  \begin{cases}
            \vec{\emptyset}^X &\tx{if $w< 2$}\\
            \id &\tx{if $w\in[2,4[$ or $w\geq 6$}\\
            \iota' &\tx{if $w\in[4,6[$}
        \end{cases}
    \end{align*}

    This construction implies that $F$ and $G$ are 1-interleaved.
    We will not detail the proof here as it just consists of proving that every possible diagram of the definition of $\epsilon$-interleaved~\Cref{def:epsilon-close} is commutative.
    Instead, we just provide the summary diagram below:
    \begin{center}
\adjustbox{scale=0.75}{%
\begin{tikzcd}
F_0=\emptyset \arrow[r, "\vec{\emptyset}", tail]                                                    & F_1=X \arrow[rd, "\iota" description, tail] \arrow[rr, "\iota", tail] &                                                                  & F_3=X' \arrow[rd, "id" description, tail] \arrow[rr, "\iota'", tail] &                                                                          & F_5=X'' \arrow[rd, "id" description, tail] \arrow[rr, "id", tail] &                                                                  & F_7=X'' \\
G_0 = \emptyset \arrow[rr, "\vec{\emptyset}", tail] \arrow[ru, "\vec{\emptyset}" description, tail] &                                                                       & G_2=X' \arrow[ru, "id" description, tail] \arrow[rr, "id", tail] &                                                                      & G_4=X' \arrow[ru, "\iota'" description, tail] \arrow[rr, "\iota'", tail] &                                                                   & G_6=X'' \arrow[ru, "id" description, tail] \arrow[r, "id", tail] & \dots  
\end{tikzcd}
}

    \end{center}

    Using the definition of ranging-sets, we obtain the following characterizations:
    \begin{align*}
    R^\F_F(3\leq7) &= \left\{ (B',F_3)\ /\ \exists x\leq 3,~y\geq 7\ \tx{and}\ B\ /\
    \begin{cases}
      \ (B,F_x)& \in\F\ \tx{with}\ B'=\Phi F_x^3(B)\\
      F_3^{y}(B',F_{3})& \in\F
    \end{cases}  \right\}
        \\
    &= \left\{ (B',X')\ /\ \exists x\leq 3\ \tx{and}\ B\ /\
    \begin{cases}
      \ (B,F_x) \in\F\ &\tx{with}\ B'=\Phi F_x^3(B)\\
      \iota'(B',X') \in\F
    \end{cases}  \right\}\\
& \quad\tx{because $F_3=X'$ and $y\geq 7$ implies that $F_3^y = \iota'$.}\\
    &= \left\{ (B',X')\ /\ \exists\ B\ /\
    \begin{cases}
      \ (B',X')\in\F &\tx{or}\ (B,X) \in\F\ \tx{with}\ B'=\Phi\iota (B) \\
      \iota'(B',X') \in\F
    \end{cases}  \right\}\\
& \quad\tx{because $F_x^3$ is either $\vec{\emptyset}^{X'}$, $\id$ or $\iota$.}\\
    &= \left\{ (B',X')\ /\
    \begin{cases}
      \ (B',X') \in\F\ &\tx{or}\ \exists\ B\ /\ B'=\Phi\iota (B)\ \tx{and}\ (B,X)\in\F\\
      \iota'(B',X') \in\F
    \end{cases}  \right\}
    \end{align*}

    Similarly, we have:
    \begin{align*}
    R^\F_G(4\leq6) &= \left\{ (B',G_4)\ /\ \exists x\leq 4,~y\geq 6\ \tx{and}\ B\ /\
    \begin{cases}
      \ (B,G_x)\in\F\ &\tx{with}\ B'=\Phi G_x^4 (B)\\
      G_4^{y}(B',G_4)\in\F
    \end{cases}  \right\}
        \\
    &= \left\{ (B',X')\ /\ \exists x\leq 4\ \tx{and}\ B\ /\
    \begin{cases}
      \ (B,G_x) \in\F\ &\tx{with}\ B'=\Phi G_x^4 (B) \\
      \iota'(B',X') \in\F
    \end{cases}  \right\}\\
& \quad\tx{because $G_4=X'$ and $y\geq 6$ implies $G_4^y = \iota'$.}\\
    &= \left\{ (B',X')\ /\
    \begin{cases}
      \ (B',X') \in\F \\
      \iota'(B',X') \in\F
    \end{cases}  \right\}\\
& \quad\tx{because $G_x^4 = \vec{\emptyset}^{X'}$ or $\id$, so $\Phi G_x^4$ is either $\vec{\emptyset}^{\Phi G_4}$ or $id$.}
    \end{align*}

    These characterizations directly imply that $R^\F_G(4\leq6)\subseteq R^\F_F(3\leq7)$.

    However, we have $(A,X)\in\F$ and $\iota'\iota(A,X)\in\F$ but $\iota(A,X)\notin\F$.
    This induces $\iota(A,X) \in R^\F_F(3\leq7)\bs R^\F_G(4\leq6)$ therefore $R^\F_G(4\leq6) \subsetneq R^\F_F(3\leq7)$.
    As those sets are finite (because they are in bijection with subsets of $\Phi F_1$ and $\Phi G_0$, which are finite), we obtain:
    $$ \left| R^\F_G(4\leq6) \right| < \left| R^\F_F(3\leq7) \right| $$

    As a result, $F$ and $G$ are $1$-interleaved but $\rho^\F_F$ and $\rho^\F_G$ are not $1$-compatible, hence $\F$ is not ranging-balanced.

\end{delayedproof}

\subsection{Proofs of~\Cref{sect:examples-constructions-hypergraphs}}
\add{
\begin{delayedproof}{prop:triangle-condition-hypergraph}
Suppose we have the following commutative diagram:
\begin{center}
\begin{tikzcd}
  \Phi H = V\sqcup E \arrow[rd, "\Phi \iota = \iota"'] \arrow[rr, "\chi"] &                          & \Phi H' = V'\sqcup E' \arrow[ld, "\Phi \iota'=\iota'"] \\
  & \Phi H'' = V''\sqcup E'' &                                                
\end{tikzcd}
\end{center}
First, we show that $\chi$ can also be considered as a monomorphism from $H = (V,E,h)$ to $H'=(V',E',h')$.
Precisely, we show that it satisfies the three properties of~\Cref{def:hypergraph-morphism}:
\begin{enumerate}
    \item $\chi(V)\subset V'$. Indeed, let $u\in V$.
    $\iota$ is a hypergraph monomorphism so we have $\iota(u) \in V''$. As $\iota(u) = \iota'(\chi(u))$, this gives $\iota'(\chi(u))\in V''$, which implies $\chi(u) \in V'$ because $\iota'$ is a hypergraph monomorphism.
    \item $\chi(E)\subset E'$ using the same argument.
    \item $u\in e \implies \chi(u)\in\chi(e)$. Indeed, let $u\in e$.
    We have $\iota(u) \in \iota(e)$ so $\iota'(\chi(u)) \in \iota'(\chi(e))$. As $\iota'$ is a monomorphism of $\cat{Hgph}^\leq_m$, we obtain $\chi(u)\in\chi(e)$.
\end{enumerate} 
Finally, we show that $\chi \in \morph(\cat{Hgph}^\leq_m)$. As $\iota$ and $\iota'$ are in $\morph(\cat{Hgph}^\leq_m)$ we have: $\chi(u)\in \chi(e) \implies \iota'(\chi(u)) \in \iota'(\chi(e)) \implies \iota(u) \in \iota(e) \implies u \in e$.
  
Similarly, if $\iota$ and $\iota'$ are monomorphisms of $\cat{Hgph}^=_m$, we obtain $\chi \in \morph(\cat{Hgph}^=_m)$: $|\chi(e)| = |\iota'(\chi(e))| = |\iota(e)| = |e|$.
\end{delayedproof}
}
\begin{proof}[\raisedtarget{mono-implies-injective} Proof of the injectivity of monomorphisms in $\cat{Hgph}$]
\ \\
Let $f$ be a monomorphism from $H=(V,E)$ to $H'=(V',E')$.
We show that $f$ is injective (that is $f_{|V}$ and $f_{|E}$ are injective).

\paragraph{$f_{|V}$ is injective:}
Let $x$ and $y$ be vertices of $V$ such that $f(x)=f(y)$.
Let $H_x=(\{x\}, \emptyset)$ (resp. $H_y=(\{y\}, \emptyset)$) be a hypergraph containing no hyperedge but only one vertex that is $x$ (resp. $y$).
Let $\iota_x$ (resp. $\iota_y$) be a hypergraph morphism from $H_x$ to $H$ that sends $x$ to $x$.

As $f(x)=f(y)$ we have that $f \circ \iota_x = f \circ \iota_y$, hence $\iota_x = \iota_y$ because $f$ is a monomorphism.
As a result $x = y$.

\paragraph{$f_{|E}$ is injective:}
Let $x$ and $y$ be hyperedges of $E$ such that $f(x)=f(y)$.
This time, we choose $v_x$ (resp. $v_y$) a vertex in $x$, and we build $H_x=(\{v_x\}, \{x\})$ (resp. $H_y=(\{v_y\}, \{y\})$).
We then use the same arguments as for $f_{|V}$: we build two simple morphisms $\iota_x$ and $\iota_y$ such that $f \circ \iota_x = f \circ \iota_y$, and we obtain $x = y$.

\end{proof}

\section{Other Experiments}\label{sect:appendix-results}
    \begin{figure}[!htb]
        \centering
        \includegraphics[width = 0.7\textwidth]{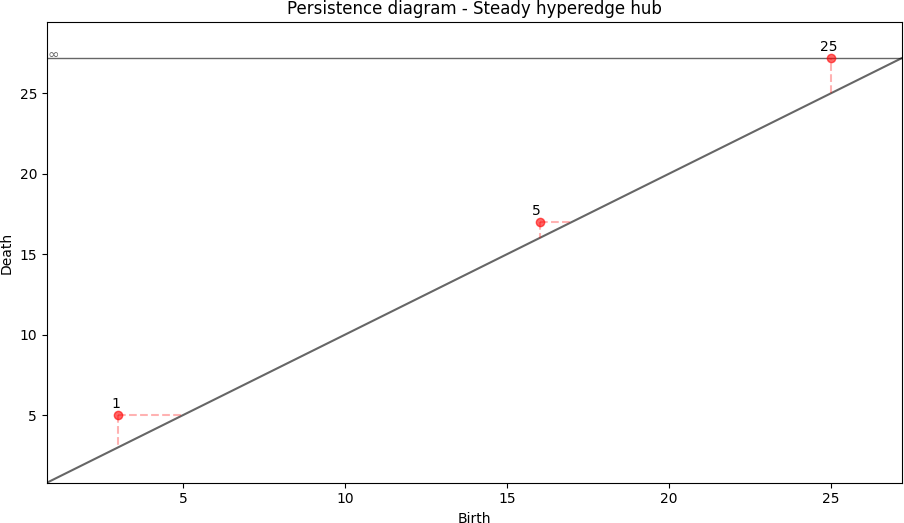}
        \includegraphics[width = 0.7\textwidth]{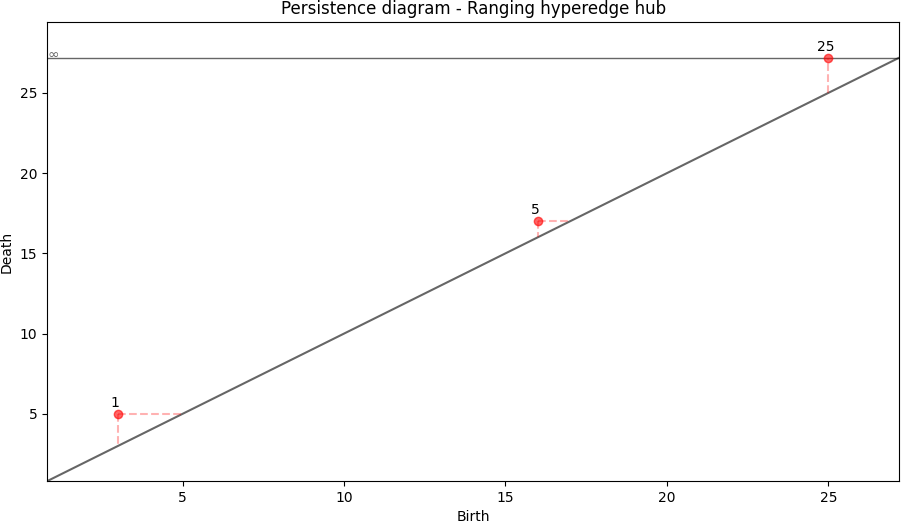}
          \caption{Steady and ranging persistence of the hub feature $\F^h$ for the scene-hypergraph filtration induced by \textit{Romeo and Juliet}. The two persistent diagrams are equal, even though $\F^h$ is not convex for $\cat{Hgph}^=_m$.}
          \label{fig:steady-ranging-hub-scene-rom_jul}
    \end{figure}

    \begin{figure}[!htb]
        \centering
        \includegraphics[width = 0.7\textwidth]{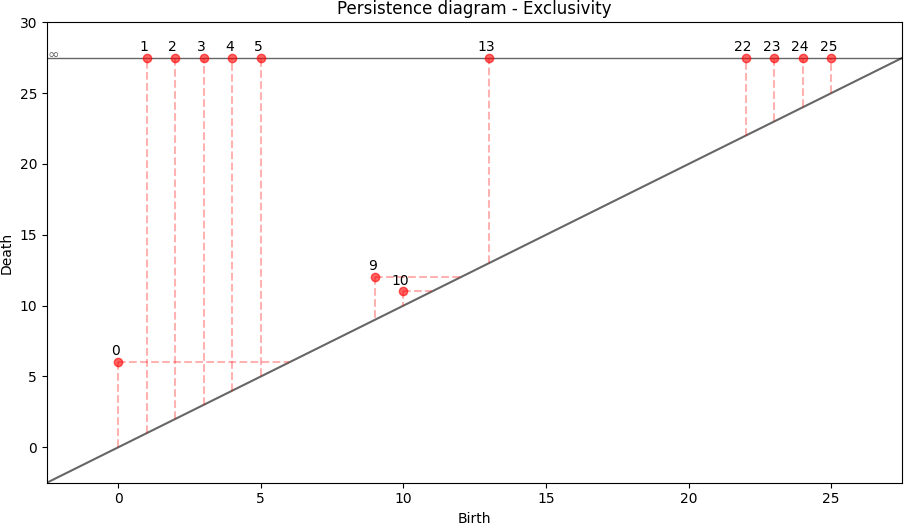}
        \includegraphics[width = 0.7\textwidth]{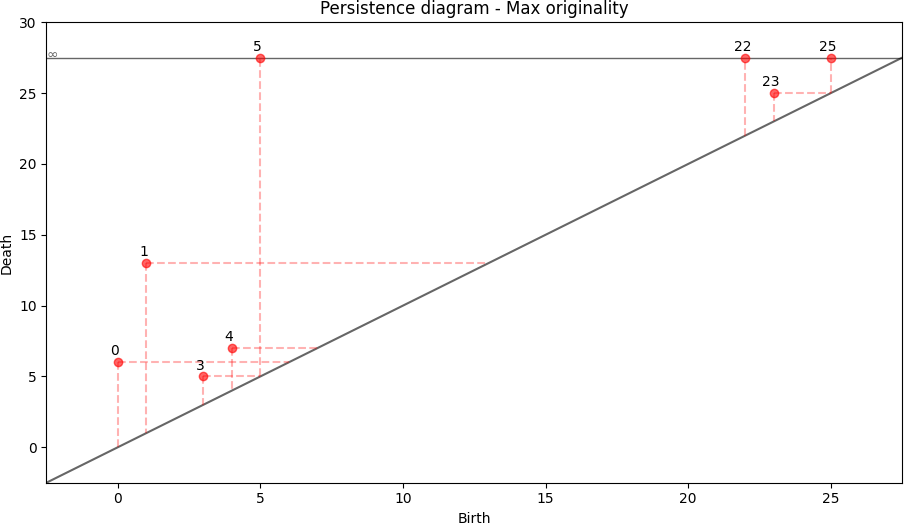}
          \caption{Persistence of the two convex features (for $\cat{Hgph}^=_m$) $\F^x$ and $\F^O$ for the scene-hypergraph filtration induced by \textit{Romeo and Juliet}.}
          \label{fig:steady-x-O-scene-rom_jul}
    \end{figure}

    \begin{figure}[!htb]
        \centering
        \includegraphics[width = 0.7\textwidth]{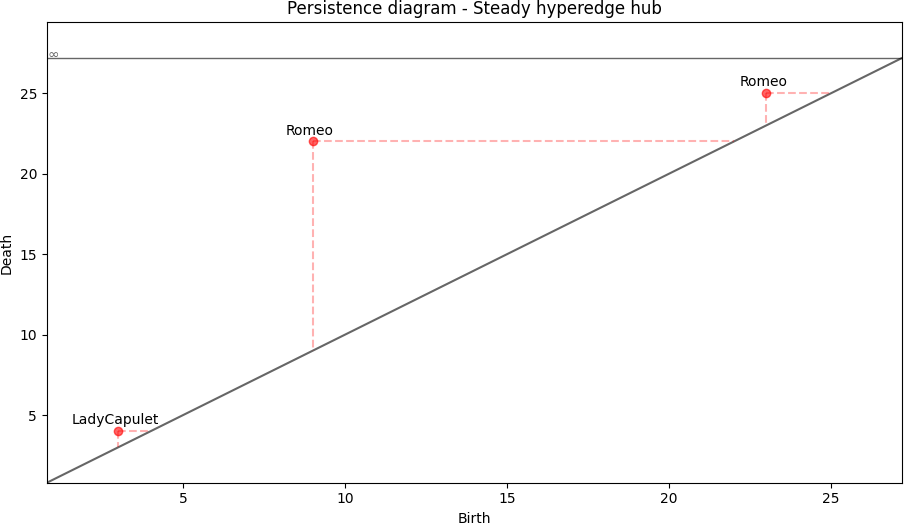}
        \includegraphics[width = 0.7\textwidth]{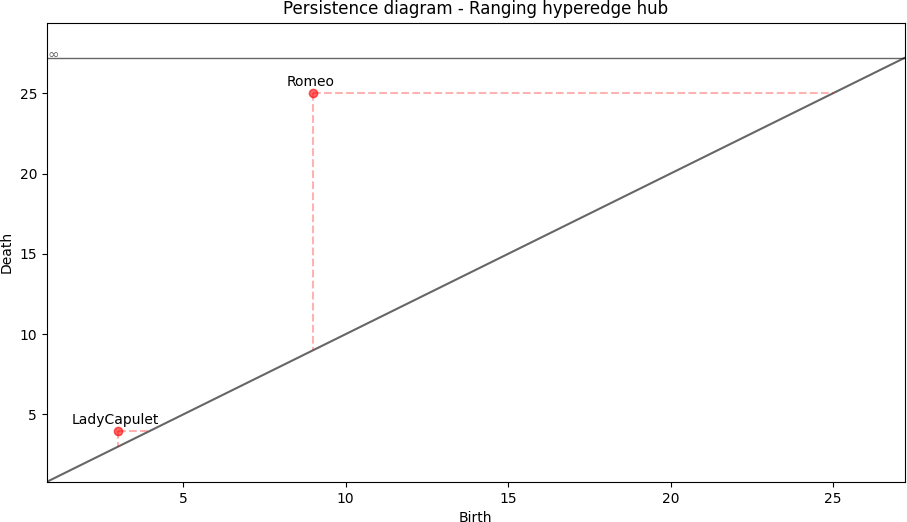}
          \caption{Steady and ranging persistence of the hub feature $\F^h$ for the character-hypergraph filtration induced by \textit{Romeo and Juliet}.}
          \label{fig:steady-ranging-hub-character-rom_jul}
    \end{figure}

    \begin{figure}[!htb]
        \centering
        \includegraphics[width = 0.7\textwidth]{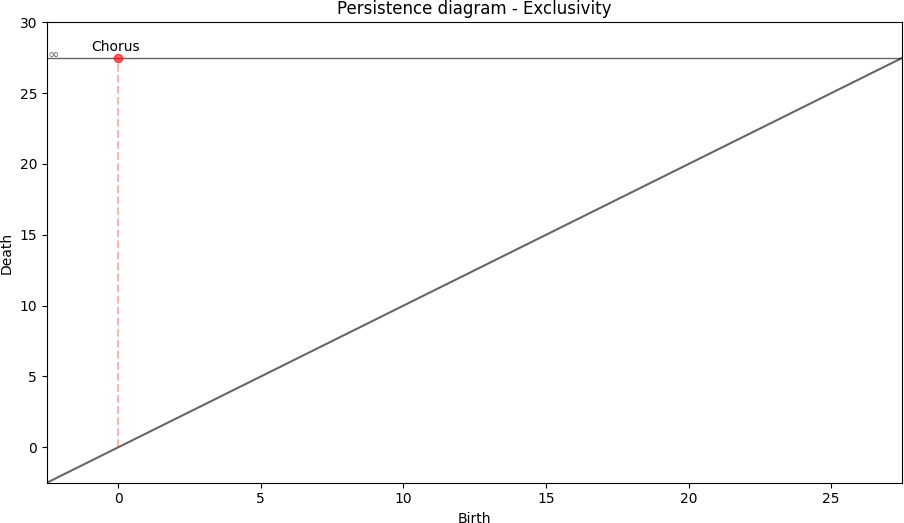}
        \includegraphics[width = 0.7\textwidth]{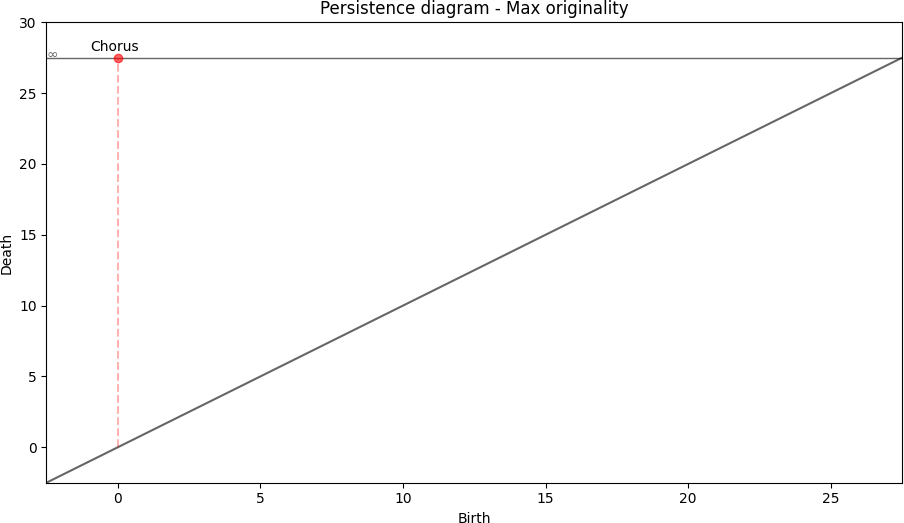}
          \caption{Persistence of the two features $\F^x$ and $\F^O$ for the character-hypergraph filtration induced by \textit{Romeo and Juliet}.
          Although these two features are not convex for $\cat{Hgph}^\leq_m$, their steady and ranging persistence diagrams for this filtration are equal.
          Note that every character-hypergraph filtration is in the category $dual(\cat{Hgph}^=)_m$ and $\F^x$ is convex for $dual(\cat{Hgph}^=)_m$ (claimed but not proved here).}
          \label{fig:steady-x-O-character-rom_jul}
    \end{figure}
\end{appendices}

\end{document}